\documentclass[12pt, final]{amsproc}
\usepackage{amssymb}
\usepackage{amscd}
\usepackage{graphicx}
\usepackage[all]{xy}

\usepackage[cp850]{inputenc}
\usepackage[mathscr]{eucal}
\usepackage[notcite,notref]{showkeys}

\begin{document}
\title{Nonlinear Centralizers in Homology\\ II. The Schatten classes}
\author{F\'{e}lix Cabello S\'{a}nchez}
\thanks{Research supported in part DGICYT projects MTM2004-02635, MTM2007-6994-C02-02, MTM2010-20190-C02-01, MTM2013-45643-C2-1-P and MTM2016-76958-C2-1-P}

\address{Departamento
de Matem\'{a}ticas and IMUEx, Universidad de Extremadura}
\address{ Avenida de Elvas, 06071--Badajoz, Spain.} \email{fcabello@unex.es}
\date{}

\noindent {\footnotesize Version \today $\,\,\mapsto$ RMI}

\bigskip

%\maketitle

%\markboth{Nonlinear Centralizers in Homology. II}{The Schatten classes}
 
\theoremstyle{plain}
\newtheorem{problem}{Problem}
\newtheorem{theorem}{Theorem}
\newtheorem{proposition}{Proposition}
\newtheorem{corollary}{Corollary}
\newtheorem{lemma}{Lemma}
\theoremstyle{remark}
\newtheorem{definition}{Definition}
\newtheorem{example}{Example}
\newtheorem{remark}{Remark}
\newtheorem{remarks}{Remarks}
\newcommand{\R}{\mathbb{R}}
\newcommand{\N}{\mathbb{N}}
\newcommand{\K}{\mathbb{K}}
\newcommand{\C}{\mathbb{C}}
\newcommand{\Pro}{\mathbb{P}}
\newcommand{\e}{\varepsilon}
\newcommand{\yxz}{\ensuremath{0\To Y\To Z\To X\To0\:}}
\newcommand{\ywz}{\ensuremath{0\to Y\to X_\Omega \to Z\to0\:}}
\newcommand{\ywzo}{\ensuremath{0\to Y\to Y\oplus_\Omega Z_0 \to Z_0\to0\:}}

\newcommand{\Lo}{\ensuremath{L_0 }}
\newcommand{\Lp}{\ensuremath{L_p }}
\newcommand{\Loo}{\ensuremath{{L^\infty }}}
\newcommand{\loo}{\ensuremath{\ell^\infty }}
\newcommand{\lpo}{\ensuremath{\ell^p_{0} }}
\newcommand{\lqo}{\ensuremath{\ell^q_{0} }}
\newcommand{\lp}{\ensuremath{\ell^p}}
\newcommand{\To}{\ensuremath{\longrightarrow}}

\newcommand{\Hi}{\mathscr H}
\newcommand{\xoy}{x\otimes y}

\newcommand{\W}{\mathscr W}

\newcommand{\seqF}{\ensuremath{0\to Y\to Y\oplus_FZ\to Z\to0\:}}

\newcommand{\seqG}{\ensuremath{0\to Y\to Y\oplus_GZ\to Z\to0\:}}

\newcommand{\F}{\ensuremath{F:Z\to Y\:}}

\newcommand{\Lin}{\ensuremath{L:Z\to Y\:}}

\newcommand{\B}{\ensuremath{B:Z\to Y\:}}

\def\Ext{\operatorname{Ext}}
\def\Hom{\operatorname{Hom}}
\def\PB{\operatorname{PB}}
\def\PO{\operatorname{PO}}
\def\El{\operatorname{Ext}}
\def\dim{\operatorname{dim}}
\def\supp{\operatorname{supp}}
\def\Bil{\operatorname{\mathfrak B}}
\def\dist{\operatorname{dist}}
\def\M{\operatorname{M}}

\def\tr{\operatorname{tr}}
\def\rk{\operatorname{rk}}
\def\spa{\operatorname{span}}

\bibliographystyle{plain}

\begin{abstract}
An extension of $X$ by $Y$ is a short exact sequence of quasi Banach modules and homomorphisms $0\To Y\To Z\To X\To 0$. When properly organized all these extensions constitute a linear space denoted by $\Ext_B(X,Y)$, where $B$ is the underlying (Banach) algebra. In this paper we ``compute'' the spaces of extensions for the Schatten classes when they are regarded in its natural (left) module structure over $B=B(\mathscr H)$, the algebra of all operators on the ground Hilbert space. Our main results can be summarized as follows:
$$
\Ext_B(S^p,S^q)=\begin{cases}0 &\text{if $0<q<p\leq\infty$ or $p=q=\infty$},\\
\Ext_\C(S^1,\C)  &\text{if $q=p$ is finite},\\
\Ext_\mathbb C(\mathscr H) &\text{if $0<p<q\leq\infty$}.
\end{cases}
$$
In the first case, every extension $0\To S^q\To Z\To S^p\To 0$ splits and so $X=S^q\oplus S^p$. In the second case, every self-extension of $S^p$ arises (and gives rise) to a minimal extension of $S^1$ in the quasi Banach category, that is, a short exact sequence $
0\To \C%\stackrel{\imath}
\To M
%\stackrel{\pi}
\To S^1\To 0
$. In the third case, each extension corresponds to a ``twisted Hilbert space'', that is, a short exact sequence $
0\To \mathscr H
\To T
%\stackrel{\pi}
\To \mathscr H\To 0
$.
Thus, the subject of the paper is closely connected to the early ``three-space'' problems studied (and solved) in the seventies by Enflo, Lindenstrauss, Pisier, Kalton, Peck, Ribe, Roberts, and others.
\end{abstract}

\maketitle

\markboth{Nonlinear Centralizers in Homology. II}{The Schatten classes}

\section{Introduction}

%\subsection{Purpose}
Let $A$ be a Banach algebra and let $X$ and $Y$ be quasi Banach modules over $A$.
% Let $X$ and $Y$ be quasi-Banach modules over a fixed Banach algebra $A$.
  An extension of $X$ by $Y$ is a short exact sequence of quasi Banach modules and homomorphisms
$$
\begin{CD}
 0@>>> Y@>>> Z@>>> X@>>> 0
\end{CD}
$$
Less technically we may think of $Z$ as a module containing $Y$ as a closed submodule in such a way that $Z/Y$ is (isomorphic to) $X$. The extension is said to be trivial (or to split) if $Y$ is complemented in $Z$ through a homomorphism. This roughly means that $Z$ is the direct sum $Y\oplus X$ and the arrows are the obvious ones.

When properly classified and organized the extensions of $X$ by $Y$ constitute a linear space, denoted by $\Ext_A(X,Y)$, whose zero is the trivial extension. When $Y=X$ we just write $\Ext_A(X)$.

While the homomorphisms between a given couple of modules display the most basic links between them, extensions reflect
much more subtle connections, often in an encrypted or disguised form.

\subsection{Summary}
In this paper we deal with extensions of the Schatten clases $S^p$ for $0<p\leq\infty$ when these are regarded as modules over $B=B(\mathscr H)$, the algebra of all (linear, bounded) operators on the underlying Hilbert space $\mathscr H$. Thus we are concerned with short exact sequences of (say left) $B$-modules
\begin{equation}\label{spsq}
\begin{CD}
 0@>>> S^q@>>> Z@>>> S^p@>>> 0
\end{CD}
\end{equation}
We perform a rather complete study of such objects.
The leading idea of the paper is that each extension of the form (\ref{spsq}) corresponds to a ``centralizer'' from $S^p$ to $S^q$, that is, a mapping that, despite of not being linear nor bounded, ``almost commutes'' with the outer products in the sense of obeying an estimate
$$
\|\Phi(af)-a\Phi(f)\|_q\leq M\|a\|_B\|f\|_p,
$$ 
for some $M$, all $a\in B$ and every finite-rank $f$.

\medskip

Let us describe the organization of the paper and highlight its main results.
This Section contains, apart from this general introduction, a list of notations and conventions that will be used along the paper.

In Section~\ref{centralizers} we give the definition of a centralizer and explain the correspondence between centralizers and extensions. We also provide some simplifications and the main ``classical'' examples that substantiate the paper.

Section~\ref{qlessp} is entirely devoted to proving that $\Ext_B(S^p,S^q)=0$ for $0<q<p\leq \infty$. The proof depends on Raynaud's representation of the ultrapowers of the Schatten classes and exploits the rather vague idea that a good enough description of the operators on ultrapowers often gives information about the extensions of the base spaces. 

In Section~\ref{sec:isom} we prove that the space $\Ext_B(S^p,S^q)$ depends only on the parameter $q^{-1}-p^{-1}$.
In particular, we obtain that $\Ext_B(S^p)$ is basically independent on $0<p<\infty$. What is quite useful since, while some properties of $\Ext_B(S^p,S^q)$ are easier to handle when $q\geq 1$, other properties are much more easy when $p<1$. The results are presented first for centralizers and then in the classical homological way, using the Hom-Ext sequences.

Section~\ref{qgp} studies extensions (\ref{spsq}) for $q\geq p$. By using almost summing operators it is shown that every ``twisted Hilbert space'', that is, every extension of Banach spaces and operators
$$
\begin{CD}
 0@>>> \Hi@>>> T@>>> \Hi@>>> 0,
\end{CD}
$$
gives rise to an extension of $S^p$ by $S^q$ in the category of left (or right) $B$-modules. And, conversely, every such an extension induces a twisted Hilbert space, which arises as its ``spatial part''. When $q>p$ these processes are each inverse of the other and provide natural isomorphism between $\Ext_B(S^p,S^q)$ and $\Ext_\C(\Hi,\Hi)$.

In Section~\ref{sec:Kspaces} we obtain the surprising result that, for each $0<p<1$, there are nontrivial extensions of quasi Banach spaces
$$
\begin{CD}
 0@>>> \C@>>> E @>>> S^p@>>> 0
\end{CD}
$$
Thus $S^p$ is not a $\mathcal{K}$-space for $0<p<1$ in striking contrast to the commutative situation where $\ell^p$ is the prototypical $\mathcal{K}$-space!

We also relate minimal extensions of $S^1$ to centralizers and we show that quasilinear functions $\phi:S^1\To\C$ and ``self-centralizers'' on $S^p$ are two faces of the same coin.

Finally, in Section~\ref{sec:bicentralizers}, we study ``bicentralizers'', that is, those centralizers associated to bimodule extensions. We show that all bicentralizers from $S^p$ to $S^q$ are trivial unless $p=q$. As for  ``self-bicentralizers'' on $S^p$ we complete a result by Kalton showing that every symmetric $\ell^\infty$-centralizer on $\ell^p$ extends to a bicentralizer on $S^p$, regardless of the value of $p$.
\medskip

To sum up, we have:
$$
\Ext_B(S^p,S^q)=\begin{cases}0 &\text{if $0<q<p\leq\infty$ or $p=q=\infty$},\\
\Ext_\C(S^1,\C)  &\text{if $q=p$ is finite},\\
\Ext_\mathbb C(\mathscr H) &\text{if $0<p<q\leq\infty$}.
\end{cases}
$$
We believe that even the existence of a nontrivial extension of Banach modules of the form
$$
\begin{CD}
 0@>>> K@>>> Z@>>> S^1@>>> 0,
\end{CD}
$$
which corresponds to the choice $(p,q)=(1,\infty)$, is quite surprising.
It is remarkable that the results of the present paper are so cleanly connected with the early ``three space'' problems. We refer the reader to \cite[Chapter 5]{kpr}, \cite[Chapter 3]{cg}, \cite[Chapter 14]{bl}, \cite[Section 4]{k-handbook} or \cite[Sections 8 and 9]{kal-mon} for basic information on the topic.

\subsection{Background}
The study of the module structures of noncommutative $L^p$ spaces built over a general von Neumann algebra $\mathcal M$ goes back to their inception. However, the computation of the spaces of homomorphisms, which plays a r\^ole in this paper, is relatively recent; see \cite{js}.

Not much is known about the corresponding spaces of extensions $\Ext_\mathcal M(L^p,L^q)$ for general $\mathcal M$.

The notion of a centralizer is an invention of Kalton, 
who introduced it in the memoir \cite{kalcom}, isolating a property shared
by most ``derivations'' appearing in interpolation theory.

By following ideas of  \cite{kal-diff} it is proved in \cite{ccgs} that $\Ext_\mathcal M(L^p)\neq 0$ for every (infinite-dimensional) $\mathcal M$ and other related results. Some loose ends were tied up in \cite{c-noncomm}. Not surprisingly, these papers make heavy use of complex interpolation theory.

The approach of this paper also originates in Kalton's work. Indeed, the idea of representing extensions by centralizers is already in \cite{kalcom}. Even if the connection between centralizers and extensions is deliberately neglected in both \cite{k-trace} and \cite{kal-diff}, these papers should be considered as the first serious studies on self-extensions of the Schatten classes within the category of quasi Banach bimodules over $B$. The paper \cite{suarez2014} contains some remarks on the structure of these extensions.

The commutative situation is settled in \cite{c} with somewhat different techniques. Considering the usual Lebesgue spaces $L^p=L^p(\mu)$ for an arbitrary measure $\mu$ as $L^\infty$-modules with ``pointwise'' multiplication we have $\Ext_{L^\infty}(L^p,L^q)=0$ when $p\neq q$ and  $\Ext_{L^\infty}(L^p)=  \Ext_{L^\infty}(L^1)$ for every $p\in(0,\infty)$. The preceding identity had been proved  for $p\in(1,\infty)$ in \cite{kalcom}.
Apologizing in advance for the pun, the present paper can be seen as a ``crossed product'' of \cite{kal-diff} and \cite{c}.

Some authors consider a more restrictive notion of extension by requiring the splitting in the quasi Banach category (``no linear obstruction to split''). This leads to the study of the amenability of the underlying algebra, a major theme in the homology of Banach algebras \cite{hel}. Although we have not pursued this point, the results of this paper suggest that if (\ref{spsq}) splits as an extension of quasi Banach spaces, then so it does as an extension of quasi Banach modules over $B$, which is easy to prove, and well-known,  for $q\geq 1$.

Finally, we refer the reader to \cite{willian2018} for a quite interesting study of extensions in the related setting of operator spaces.

\subsection{Notation and some general conventions}\label{dnc}
\begin{itemize}
\item The ground field is $\mathbb C$, the complex numbers.
\item $\mathscr H$ is the underlying separable Hilbert space where our operators act and $\langle \cdot |\cdot \rangle$ is the scalar product in $\mathscr H$.
\item $B=B(\mathscr H)$ is the Banach algebra of all (linear, bounded) operators on $\mathscr H$. A ``projection'' is a self-adjoint idempotent of $B$. The ideal of finite rank operators is denoted by $\frak F$. The ideal of compact operators is denoted by $K$.
\item $L(\mathscr H)$ is the algebra of all (not necessarily continuous) linear endomorphisms of $\mathscr H$.
\item  If $x\in\mathscr H$ and $y\in Y$, then $x\otimes y:\mathscr H\To Y$ is the rank-one operator given by $h \mapsto \langle h|x\rangle y$. 

\item The weak operator topology (WOT) in $B$ is that generated by the seminorms $u\mapsto |\langle y | u(x)\rangle|$, with $x,y\in\mathscr H$.

\item If $V$ is any linear (respectively, quasinormed) space, then $V^\star$ (respectively, $V'$) denotes the space of linear functionals (respectively, bounded linear functionals) on $V$.
 The symbol $^*$ is reserved for the Hilbert space adjoint.

    \item Let $U,V$ and $W$ be arbitrary sets and  $\varphi: U\To V$ any mapping. We define $\varphi_\circ: U^W\To V^W$ by $\varphi_\circ(f)=\varphi\circ f$. Similarly, $\varphi^\circ: W^V \To W^U$ is defined as $\varphi^\circ(f)= f\circ \varphi$. The identity on $U$ is denoted by ${\bf I}_U$.

\item Let $v$ be a finite rank endomorphism of the linear space $V$ (no topology is assumed). Then the trace of $v$ is given by $\tr v=\sum_{i=1}^n v_i^\star(v_i)$ provided $v=\sum_{i=1}^n v_i^\star\otimes v_i$, with $v_i^\star\in V^\star, v_i\in V$. The trace does not depend on the given representation since, after the identification of the finite rank endomorphisms of $V$ with $V^\star\otimes V$, the trace is nothing different from the linearization of the obvious bilinear function $V^\star\times V\To\mathbb C$. If $u$ is any endomorphism of $V$ and $v$ has finite rank, one has $\tr(u\circ v)=\tr(v\circ u)$.

\item We use $M$ for a constant independent on operators and vectors but perhaps depending on the involved spaces  and centralizers and which may vary from line to line.

\item The distance between two maps $\phi$ and $\psi$ (acting between the same quasinormed spaces) is the least constant $\delta$ for which one has
$\|\phi(x)-\psi(x)\|\leq \delta\|x\|$ for every $x$ in the common domain.

\item A mapping $\phi:U\To V$ acting between linear spaces is said to be homogeneous if $\phi(tu)=t\phi(u)$ for every $t\in\mathbb C$ and $u\in U$.
\end{itemize}

\section{Centralizers and extensions}\label{centralizers}
In this Section we consider modules on the left unless otherwise stated.
Let $A$ be a Banach algebra that for all purposes in this paper will be a $C^*$-algebra.
A quasinormed module over $A$ is a quasinormed space $X$ together with a jointly continuous outer multiplication
$A\times X\To X$ satisfying the traditional
algebraic requirements. If the underlying space is
complete (that is, a quasi Banach space) we call it a quasi Banach module. Given
quasinormed modules $X$ and $Y$, a homomorphism $u:X\To Y$ is an
operator such that $u(ax)=au(x)$ for all $a\in A$ and $x\in X$.
Operators and homomorphisms are assumed to be continuous unless
otherwise stated. If no continuity is assumed, we speak of linear
maps and morphisms.
We use $\Hom_A(X,Y)$ for the space of homomorphisms and $\mathscr M_A(X,Y)$ for the morphisms.
 If there is no possible confusion about the underlying algebra $A$, we omit the subscript.

Quasinormed right modules and bimodules and their homomorphisms are defined in the obvious way.

 In general, $\Hom_A(X,Y)$ carries no module structure. However, if $X$ is a bimodule instead of a mere left module, then $\Hom_A(X,Y)$ can be given a structure of left module letting
$
(ah)(x)=h(xa)
$, where $h\in \Hom_A(X,Y),x\in X, a\in A$. Similarly, if $Y$ is a bimodule, then the multiplication $ha(x)=h(x)a$ makes $\Hom_A(X,Y)$ into a right module.

These structures are functorial in the obvious sense.

\subsection{Extensions}
 An extension of $X$ by $Y$ is a short exact sequence of quasi Banach modules and homomorphisms
\begin{equation}\label{s}
\begin{CD}
 0@>>> Y@>\imath >> Z@>\pi>> X@>>> 0
 \end{CD}
\end{equation}
The open
mapping theorem guarantees that $\imath$ embeds  $Y$ as a closed
submodule of $Z$ in such a way that the corresponding quotient is
isomorphic to $X$. Two extensions $0 \To Y \To Z_i \To X \To 0$
($i=1,2$) are said to be equivalent if there exists a homomorphism
$u$ making commutative the diagram
$$
\xymatrixcolsep{3.5pc}\xymatrixrowsep{0.5pc}\xymatrix{
  & & Z_1 \ar[dd]^u \ar[dr] \\
 0 \ar[r] & Y \ar[ur]\ar[dr] & & X \ar[r] & 0\\
&  & Z_2 \ar[ur]
}
$$
By the five-lemma \cite[Lemma 1.1]{hs}, and the open
mapping theorem, $u$ must be an isomorphism. We say that (\ref{s})
 is trivial if it is equivalent to the direct sum sequence
$$
\begin{CD}
 0@>>> Y@>\jmath >> Y\oplus X@>\varpi>> X@>>> 0
 \end{CD}
 $$ 
 in which $\jmath(y)=(y,0)$ and $\varpi(y,x)=x$.
  This happens if and only if (\ref{s}) splits, that is, there is a homomorphism $Z\To Y$ which
is a left inverse for the inclusion $\imath: Y\To Z$; equivalently, there
is a homomorphism $X\To Z$ which is a right inverse for the
quotient $\pi: Z\To X$. 

Given
quasi Banach modules
 $X$ and $Y$, we denote by
$\Ext_A(X,Y)$ the set of all possible extensions
(\ref{s}) modulo equivalence. When $Y=X$ we just write $\Ext_A(X)$. 

By using pull-back and push-out
constructions, it can be proved that $\Ext_A(X,Y)$ carries a ``natural'' linear
structure in such a
way that the (class of the) trivial extension corresponds to 0. This can be seen in \cite[Chapter 4, \S~9]{hs}; the approach based on injective or projective representations completely fails
dealing with quasi Banach modules since there are neither injective nor
projective objects. Thus, $\Ext_A(X,Y)=0$ means ``every extension $\yxz$ splits".

Taking $A$ as the ground field one recovers extensions in the quasi Banach space setting.

\subsection{Centralizers}\label{they}
In this paper we study extensions by means of a certain type of nonlinear (nor bounded) maps called centralizers. These offer a useful and relatively simple way to construct, describe and handle extensions that works fine with the Schatten classes.

\begin{definition}\label{def:centralizer}
Let $X$ and $Y$ be a quasinormed modules over a Banach algebra $A$ and let $W$ be another $A$-module containing $Y$ the the purely algebraic sense. Let further $\Phi: X\To W$ be a homogeneous mapping.
\begin{itemize}
\item[(a)]
We say that $\Phi$ is quasilinear from $X$ to $Y$ if, for every $f,g\in X$, the difference  $\Phi(f+g)-\Phi f-\Phi g$ belongs to $Y$ and
$
\|\Phi(f+g)-\Phi(f)-\Phi(g)\|_Y\leq Q(\|f\|_X+\|g\|_X)
$
for some constant $Q$ independent on $f,g$.
\item[(b)] We say that $\Phi$ is a left centralizer from $X$ to $Y$ if there is a constant $C$ such that for every $a\in A$ and every $ f\in X$ the difference $\Phi(af)-a\Phi(f)$ belongs to $Y$ and
$$
\|\Phi(af)-a\Phi(f)\|_Y\leq C\|a\|_A\|f\|_X.
$$
Right centralizers are defined analogously, using right module structures.
\item[(c)] Finally, $\Phi$ is said to be a bicentralizer over $A$ if it is both a left centralizer and a right centralizer. A bicentralizer obeys an estimate of the form
$$
\|\Phi(afb)-a\Phi(f)b\|_Y\leq C\|a\|_A\|f\|_X \|b\|_A.
$$
\end{itemize}
\end{definition}

If necessary, the least constants for which the preceding inequalities hold will be denoted by $Q(\Phi), L(\Phi), R(\Phi)$ and $B(\Phi)$, respectively.% denoted by $C[\Phi]$.

In this paper the underlying algebra will always be $B$ and either $X=S^p_0$ and $Y=W=S^q$ (preferably) or $X=S^p, Y=S^q$ and $W=L(\Hi)$ (if there is no choice).

Let is briefly describe the connection between centralizers and extensions. Let $X$ and $Y$ be quasi Banach spaces. Let $W$ be a linear space containing $Y$ and $X_0$ a dense subspace of $X$. Let further $\Phi:X_0\To W$ be quasilinear from $X_0$ to $Y$. Then the set
$$
Y\oplus_\Phi X_0=\{(g,f)\in W\times X: f\in X_0, g-\Phi f\in Y\}.
$$
is a linear subspace of $W\times X$ and the functional
$$
\|(g,f)\|_\Phi=\|g-\Phi f\|_Y+\|f\|_X
$$
is a quasinorm on it.
We define maps $\imath:Y\To Y\oplus_\Phi X_0$ and $\pi : Y\oplus_\Phi X_0\To X_0$ by $\imath(g)=(g,0)$ and $\pi(g,f)=f$, respectively. Clearly, $\imath$ is ``isometric'', while $\pi$ maps the unit ball of $Y\oplus_\Phi X_0$ onto that of $X_0$. Thus, we have    an exact sequence of quasinormed spaces and relatively open operators
\begin{equation}\label{yox}
\begin{CD}
0@>>> Y @>\imath>> Y\oplus_\Phi X_0@>\pi>> X_0@>>> 0.
\end{CD}
\end{equation}
%Actually only the quasilinearity of $\Phi$ is required here.
If, besides, $\Phi$ is a left centralizer, then the product $a\cdot(g,f)=(ag,af)$ 
 makes  $
Y\oplus_\Phi X_0$ into a quasinormed module over $A$ and the arrows in
the preceding diagram become homomorphisms. Indeed,
$$%\begin{align*}
\|a(g,f)\|_\Phi= \|ag-\Phi(af)\|_Y+\|af\|_X
=  \|ag-a\Phi f+ a\Phi f-\Phi(af)\|_Y+\|af\|_X
\leq M \|a\|_A\|(g,f)\|_\Phi.
$$%\end{align*}
Let $Z_\Phi$ be the completion of $Y\oplus_\Phi X_0$. This is a quasi Banach module and there is a unique surjective homomorphism $Z_\Phi\To Z$ extending the quotient in (\ref{yox}) we denote again by $\pi$. We have a commutative diagram
$$%\begin{equation}\label{compl}
\xymatrixcolsep{3.5pc}\xymatrix{%\begin{CD}
 0 \ar[r] & Y \ar[r] \ar@{=}[d] & Y\oplus_\Phi X_0\ar[r] \ar[d] & X_0\ar[r] \ar[d] & 0\\
 0 \ar[r] & Y \ar[r] & Z_\Phi \ar[r] & X\ar[r] & 0\\
}%\end{CD}
$$%\end{equation}
in which the vertical arrows are inclusions and the horizontal rows are exact. We will always refer to the lower row in this diagram as the extension (of $X$ by $Y$) induced by $\Phi$.

It is easily seen that two centralizers $\Phi$ and $\Gamma$ (acting between the same sets, say $X_0$ and $W$) induce equivalent extensions if and only if there is a morphism $\alpha:X_0\To W$ such that the differences $\Phi(f)-\Gamma(f)- \alpha(f)$ fall in $Y$ and satisfy an estimate $\|\Phi(f)-\Gamma(f)- \alpha(f)\|_Y\leq M\|f\|_X$ for some $M$ and all $f\in X_0$. 

In this case we say that $\Phi$ and $\Gamma$ are equivalent centralizers, and we write $\Gamma\sim \Phi$.
If the preceding inequality holds for $\alpha=0$, that is, if $\Phi-\Gamma$ is bounded from $X_0$ to $Y$, then we say that $\Phi$ and $\Gamma$ are strongly equivalent and we write $\Gamma\approx\Phi$.

In particular the extension induced by $\Phi$ is trivial if and only if there is a morphism $\alpha: X_0\To W$ such that $\Phi-\alpha$ takes values in $Y$ and satisfies  $\|\Phi(f)-\alpha(f)\|_Y\leq M\|f\|_X$ for some $M$ and all $f\in X_0$. In this case we say that $\Phi$ is a trivial centralizer.

\subsection{The Schatten classes $S^p$}
%We now move to the concrete modules we shall deal with.
For $p\in (0,\infty)$, let $\ell^p$ denote quasi Banach space of (complex) sequences $(t_n)$ for which the quasinorm $|(t_n)|_p=\left(\sum_n|t_n|^p\right)^{1/p}$ is finite.

Let $f$ be a compact operator on the Hilbert space $\mathscr H$. The singular numbers of $f$ are the eigenvalues of $|f|=(f^*f)^{1/2}$ arranged in decreasing order and counting multiplicity.
The Schatten class $S^p$ consists of those operators on $\mathscr H$ whose sequence of singular numbers $(s_n(f))$ belongs to $\ell^p$. It is a quasi Banach space under the quasinorm $\|f\|_p=|(s_n(f))|_p$.
Each $f\in S^p$ has an expansion $f=\sum_ns_n\: x_n\otimes y_n$, where $s_n$ are its singular numbers and $(x_n)$ and $(y_n)$ are orthonormal sequences in $\mathscr H$. This is called a Schmidt representation of $f$.
$S^p$ is a quasi Banach bimodule over $B$ in the obvious way: given $f\in S^p$ and $a,b\in B$ one has $afb\in S^p$ and $\|afb\|_p\leq \|a\|_B\|f\|_p\|b\|_B$. The submodule of finite rank operators is denoted by $S^p_0$.
The structure of homomorphisms between Schatten classes is fairly simple. Indeed, one has
\begin{equation}\label{eq:HomSpSq}
\Hom_B(S^p,S^q)=\begin{cases} S^r & \text{if $0<q<p<\infty$, where $p^{-1}+r^{-1}=q^{-1}$;}\\
B & \text{if $p\leq q$.}
\end{cases}
\end{equation}
This should be understood as follows: each operator $g$ in the right-hand side defines a homomorphism $\gamma: S^p\To S^q$ by multiplication on the right $\gamma(f)=fg$. Moreover, the norm of $g$ in the corresponding space equals $\|\gamma: S^p\To S^q\|$ and every homomorphism arises in this way.
%All this can be seen in Simon's monograph \cite{simon}.

It will be convenient at some places to consider right module structures. We indicate this just by putting the (algebra) subscript on the right. Thus, for instance, $\Hom(X,Y)_A$ is the space of homomorphisms of right modules from $X$ to $Y$, which are assumed to be (quasinormed) right modules over $A$. The meaning of $\mathscr M(X,Y)_A, \Ext(X,Y)_A$ or ``right centralizer'' should be clear.

It it worth noticing that the right module structure of Schatten classes is ``isomorphic'' to the left one throughout the involution: $fa=(a^* f^*)^*$. Thus, for instance, if $u:S^p\To B$ is a morphism of left (respectively, right) modules, then we obtain a morphism of right (respectively, left) modules thus: $f\mapsto (u(f^*))^*$. The same formula can be used to exchange left and right homomorphisms, centralizers, and the like. We will use this fact without further mention.

\begin{lemma}\label{mor}
\begin{itemize}
\item[(a)] $\frak F$ is a projective left (or right) module over $B$ in the pure algebraic sense:
if $\varpi: X\To Y$ is a surjective morphism of left modules, then every morphism $\alpha: \mathfrak F\To Y$ lifts to $X$ in the sense that there is another morphism of left modules $\widehat{\alpha}: \mathfrak F\To Y$ such that $\alpha=\varpi\circ\widehat{\alpha}$.

%\item[(b)] The same is true replacing ``left'' by ``right'' everywhere in {\rm (a)}.

    \item[(b)] $\mathscr M(\frak F,B)_B=L(\mathscr H)$ in the sense that for every morphism of right modules $\alpha:\frak F\To B$ there is a unique linear endomorphism $\ell$ of $\mathscr H$ such that $\alpha(f)=\ell\circ f$ for every $f\in\frak F$.
 \item[(c)] Similarly, $\mathscr M_B(\frak F,B)=L(\mathscr H)$ in the sense that for every morphism of left modules $\alpha:\frak F\To B$ there is a unique linear endomorphism $\ell$ of $\mathscr H$ such that $\alpha(f)=(\ell\circ f^*)^*$ for every $f\in\frak F$.
\item[(d)] Let $\ell:\mathfrak F\To\mathbb C$ be a linear map such that for each fixed $y\in \mathscr H$ one has $\ell(x\otimes y)\To0$ as $x\To0$ in $\mathscr H$. Then there is a linear endomorphism $L$ of $\mathscr H$ such that $\ell(f)=\tr(L\circ f)$ for all $x,y\in \mathscr H$.
\end{itemize}
\end{lemma}

\begin{proof}
(a) Of course, $B$ is a projective left (or right) module. Also, $\mathscr H$ is a left module under the obvious action $(a,h)\mapsto a(h)$, while $\Hi'$ is a right module under the dual action: $\langle h'a, h\rangle= \langle h', ah\rangle$.

Let us see that $\mathfrak{F}$ is a projective left module.
Fix any norm one $\eta\in\mathscr  H$. Then the map $\eta\otimes-: \mathscr H\To  B$ given by $h\mapsto \eta\otimes h$ is an injective (homo)morphism. The evaluation map $\delta_\eta:B\To\mathscr  H$ given by $\delta_\eta(u)=u(\eta)$ is a morphism of left modules and, quite clearly,
$\delta_\eta\circ(\eta\otimes-)={\bf I}_\Hi$.

Being a direct factor in $B$, $\mathscr H$ is projective too.

On the other hand, $\frak F\cong\mathscr H\otimes_\C \mathscr H'$ (as bimodules). If $I$ is a Hamel basis for $\mathscr H'$, then $\mathscr H'$ is linearly isomorphic to the direct sum $\bigoplus_{I}\mathbb C$. Combining, we have isomorphisms of left modules
$$
\frak F\cong  \mathscr  H\otimes_\C\mathscr  H'\cong \Hi \otimes_\C \left(\bigoplus_{I}\mathbb C\right) = \bigoplus_{I}\left(\Hi\otimes_\C \C \right)=
\bigoplus_{I} \mathscr H,
$$
and a direct sum of projective modules is again projective.

The proof that $\mathfrak{F}$ is right projective is similar: first we embed $\Hi'$ into $B$ fixing a normalized $y\in\Hi$ and then sending each $h'\in\Hi'$ into the rank-one operator $h'\otimes y$. This is a homomorphism of right modules. The corresponding projection is given as follows: given $f\in B$ one considers the Banach (not Hilbert) space adjoint $f':\Hi'\To \Hi'$ and the evaluation at $y'=\langle-|y\rangle$. This shows that $\Hi'$ is right projective. Now, if $J$ is a Hamel basis of $\Hi$, then, as a bimodule, 
$\mathfrak{F}$ is isomorphic to the direct sum $\bigoplus_{J} \mathscr H'$, which is projective.

(b) is very easy. Take $x,y\in \mathscr H$, with $\|x\|= 1$. Then
$
\alpha(x\otimes y)= \alpha((x\otimes y)(x\otimes x))=
(\alpha(x\otimes y))(x\otimes x).
$
Hence there is $z=z(x,y)\in\mathscr H$ such that $\alpha(x\otimes y)=x\otimes z$. It is easily seen that $z$ does not depend on the first variable while it depends linearly on the second one. Thus the rule $\ell(y)=z$ is an endomorphism of $\mathscr H$. Quite clearly one has $\alpha(f)=\ell \circ f$ when $f$ has rank one and the same is true for every $f\in\frak F$.

(c) is just the left version of (b).

(d) Fix $y\in\mathscr H$. The hypothesis implies that $x\mapsto \ell(x\otimes y)$ is a continuous, conjugate-linear functional on $\mathscr H$ and by Riesz representation theorem there is $z\in\mathscr H$ such that $\ell(x\otimes y)=\langle z |x \rangle$. Putting $z=L(y)$ we obtain a transformation of $\mathscr H$ which is easily seen to be linear. And since $\ell(x\otimes y)= \langle L(y) |x \rangle =\tr(x\otimes L(y))= \tr(L\circ(x\otimes y))$ we are done.
\end{proof}

The following result is a slight improvement of Kalton's \cite[Proposition 4.1]{k-trace}, with a different proof. %It allows us to skip quasilinearity  spare the verification of the first part of Definition~\ref{def:cen}.

\begin{lemma}\label{lem:spare}
With the same notations of Definition~\ref{def:centralizer}, let us assume that $A=B$ and that $X$ is either $S^p$ or $S^p_0$ for some $0<p\leq\infty$.
Then every homogeneous left (or right) centralizer $\Phi: X\To W$ is quasilinear.
\end{lemma}

\begin{proof}
We write the proof for left modules and $X=S^p$.
The key point is that if $f,g\in S^p$, then $h=(f^*f+g^*g)^{1/2}$ belongs to $S^p$ and one has $f=ah, g=bh$ for certain contractive $a, b\in B$ --whose initial projections agree with the final projection of $h$, if you want. Indeed one may take 
 $a=f(f^*f+g^*g)^{-1/2}$ which is contractive by Schmitt's
\cite[Lemma 2.2(c)]{schmitt}: just set $T=f^*f, S=h$ and follow Schmitt's notations.

 As for the quasinorm of $h$ we have
\begin{align*}
\|(f^*f+g^*g)^{1/2}\|_p&= \|f^*f+g^*g\|_{p/2}^{1/2}\leq \Delta_{p/2}^{1/2}\left( \|f^*f\|_{p/2}+\|g^*g\|_{p/2}	 \right)^{1/2}\\
&\leq  \Delta_{p/2}^{1/2}\left( \|f^*f\|_{p/2}^{1/2}+\|g^*g\|_{p/2}^{1/2}	\right)= \Delta_{p/2}^{1/2}(\|f\|_p+\|g\|_p),
\end{align*}
where $\Delta_{r}$ denotes the ``modulus of concavity'' of $S^{r}$, that is, $\Delta_r=2^{1/r-1}$ for $r<1$ and $\Delta_r=1$ for $r\geq 1$.
Now, if $\Phi:S^p\To W$ is a centralizer from $S^p$ to $Y$, and $f,g\in S^p$, then
\begin{align*}
\|\Phi(f+g)&-\Phi f-\Phi g\|_Y=\|\Phi((a+b)h)-\Phi(ah)-\Phi(bh)\|_Y\\
&\leq \Delta_Y^2(\|\Phi((a+b)h)-(a+b)\Phi h\|_Y+\|a\Phi h-\Phi(ah)\|_Y+\|b\Phi h-\Phi(bh)\|_Y)\\
&\leq \Delta_Y^2 L(\Phi) \big{(} (\|a+b\|_B+ \|a\|_B+\|b\|_B) \|h\|_p   \big{)} \\
&\leq  4 \Delta_Y^2\Delta_{p/2}^{1/2}L(\Phi)(\|f\|_p+\|g\|_{p}),
\end{align*}
and we are done.
\end{proof}

It is clear that, for fixed $Y$ and $p$, the quasilinear constant of $\Phi$ is controlled by the centralizer constant.

\begin{corollary}\label{all}
Every extension of $S^p$ by an arbitrary quasi Banach left (respectively, right) module $Y$ comes from a left (right) centralizer $\Phi:S^p_0\To Y$, up to equivalence.
\end{corollary}

\begin{proof}We consider the case of left modules. Let
$\xymatrixcolsep{2.5pc}\xymatrixrowsep{0.5pc}\xymatrix{
0 \ar[r] & Y  \ar[r]^-\imath & Z \ar[r]^-\pi & S^p \ar[r] & 0 
}$ be an extension of quasi Banach modules over $B$. With no serious loss of generality we may assume $Y=\ker \pi$ and that $\imath$ is just the inclusion. Putting $Z_0=\pi^{-1}[S_0^p]$ we have the following commutative diagram
$$
\xymatrixcolsep{3.5pc}\xymatrixrowsep{0.5pc}\xymatrix{
  & & Z_0 \ar[dd] \ar[r] & S^p_0 \ar[rd] \ar[dd] \\
 0 \ar[r] & Y \ar[ur]\ar[dr] & & & 0\\
&  & Z \ar[r] & S^p \ar[ur]
}
$$
where the vertical arrows are plain inclusions. We shall show there is a centralizer $\Phi: S_0^p\To Y$ and an isomorphism of quasinormed modules $u$ making commutative the diagram
$$
\xymatrixcolsep{3.5pc}\xymatrixrowsep{0.5pc}\xymatrix{
  & & Y\!\oplus_\Phi \! S_0^p \ar[dd]^u \ar[dr] \\
 0 \ar[r] & Y \ar[ur]\ar[dr] & & S^p_0 \ar[r] & 0\\
&  & Z_0 \ar[ur]
}
$$
This obviously implies that $u$ extends to an isomorphism between the completion of $Y\oplus_\Phi S_0^p$ and $Z$ fitting in the corresponding diagram.

One can construct such a $\Phi$ as follows. First, let $B:S^p\To Z$ be a homogeneous bounded section of the quotient map $\pi:Z\To S^p$, which exists because $\pi$ is open. Clearly, $B(f)\in X_0$ if $f\in S^p_0$. On the other hand, by Lemma~\ref{mor}(a), there is a morphism $\alpha :S^p_0\To Z_0$ such that $\pi\circ \alpha={\bf I}_{S^p_0}$. Define $\Phi(f)=B(f)-\alpha(f)$ for $f\in S^p_0$. We have  $\pi(\Phi(f))=\pi(B(f))-\pi(\alpha(f))=0$ and so $\Phi$ takes values in $Y$. Clearly, $\Phi$ is a centralizer: given $f\in S^p_0$ and $a\in B$ one has
$$%\begin{align*}
%\|\Omega(f+g)-\Omega f-\Omega g\|_Y&=\|b(f+g)-b(f)-b(g)\|_X\leq M(\|f\|_p+\|g\|_p),\\
\|\Phi(af)-a\Phi f\|_Y= \|B(af)-aB(f)\|_X\leq M\|a\|_B\|f\|_p.
$$%\end{align*}
We define a morphism $u:Y\oplus_\Phi S^p_0\To Z_0$ by $u(y,f)=y+\alpha(f)$. This is a homomorphism in view of the bound
\begin{align*}
\|u(y,f)\|_Z&=\|y+\alpha(f)\|_Z\leq M\big{(}\|y-B(f)+\alpha(f)\|_Z+\|B(f)\|_Z\big{)}\\
&\leq M\big{(}\|y-\Phi(f)\|_Y+\|f\|_p\big{)}\leq M\|(y,f)\|_\Phi.
\end{align*}
The inverse of $u$ is
given by $v(z)=(z-\alpha(\pi(z)), \pi(z))$ for $z\in Z_0$. It is continuous since
$$
\|v(z)\|_\Phi=\|z-\alpha(\pi(z))- \Phi(\pi(z))\|_Y+\|\pi(z)\|_p= \|z-B(\pi(z))\|_Y+\|\pi(z)\|_p\leq M\|z\|_Z.
$$
This completes the proof.
\end{proof}

\subsection{Basic examples of centralizers}
%Centralizers where invented by Kalton by isolating a property shared by most ``derivations'' appearing in interpolation theory. 
% Let us present some concrete examples 
 %who isolated the crucial property
The aim of this Section is to provide the reader with a stock of centralizers substantiating the approach of the paper.
Not surprisingly, these examples are due to Kalton.

Let $x:\N\To\C$ a sequence converging to zero. The rank-sequence of $x$ is defined as
$$r_x(n)=\big{|}\{k\in\N: \text{either } |x(k)|> |x(n)| \text{ or } |x(k)|= |x(n)| \text{ and } k\leq n\}\big{|},$$
that is, $r_x(n)$ is the place that $|x(n)|$ occupies in the decreasing rearrangement of $|x|$.

Kalton proved in \cite{kalcom} that if $\varphi:\R^2_+\To\C$ is a Lipschitz function vanishing at the origin, then the map $\phi:\ell^p\To\ell^\infty$ defined by
\begin{equation}\label{eq:phi}
\phi(x)=x\:\varphi\left(\log\frac{\|x\|_p}{|x|}, \log r_x\right)
\end{equation}%$$
is a (quasilinear) self-centralizer on $\ell^p$, where $\ell^p$ is regarded as an $\ell^\infty$-module under the pointwise multiplication; this is a specialization of \cite[Theorem 3.1]{kalcom}. Moreover $
\phi$ is trivial if and only if $\varphi$ is bounded.

Actually these centralizers are symmetric in the sense that
$\phi(x\circ\sigma)= \phi(x)\circ\sigma$ when $\sigma$ is a permutation of the integers.

It is shown in \cite[Theorem 8.3]{kal-diff} that if $\phi$ is a symmetric $\ell^\infty$-centralizer on $\ell^p$, with $1<p<\infty$, then one can obtain a bicentralizer on $S^p$ as follows: for each $f\in S^p$ take a Schmidt expansion $f=\sum_n s_nx_n\otimes y_n$ and set $\Phi(f)= \sum_n t_n x_n\otimes y_n$, where $(t_n)_{n\geq 1}=\phi((x_n))$. Actually, all bicentralizers on $S^p$ arise in this way, up to strong equivalence.

Therefore, every Lipschitz function $\varphi:\R^2_+\To\C$ provides a bicentralizer on $S^p$ through the formula
\begin{equation}
\Phi(f)= \sum_n s_n\: \varphi\left(-\log\frac{s_n}{\|f\|_p}, \log n\right) x_n\otimes y_n,
\end{equation}
where $\sum_n s_nx_n\otimes y_n$ is a Schmidt expansion of $f$, at least when $1<p<\infty$. To be true Kalton had stablished this fact for all $p$ when $\varphi$ depends only on one of the variables by sheer force in \cite{k-trace}. We will complete these results in Section~\ref{sec:bicentralizers}; see Theorem~\ref{th:bic2}.

\section{The case $p>q$}

In this Section we prove that $\Ext_B(S^p,S^q)=0$ when $0<q<p<\infty$. As the reader may guess, what we actually prove is that every left centralizer
$S^p_0\To S^q$ is trivial. Theorem~\ref{qlessp} below  contains a slightly more precise statement.
To avoid annoying repetitions, througout the Section we consider only left module structures. This applies to modules, morphisms, homomorphisms and centralizers. All results remain true for right modules, with minor ajustments in the statements and proofs.

First we need to cut a given centralizer into ``small pieces'' without losing the relevant information it encodes.

%Throughout the Section all module structures refer to the left action of $B$, unless explicitly stated otherwise. 

Let $\Phi:S^p_0\To S^q$ be a left centralizer and $e\in B$ a finite-rank projection. Then we can define a centralizer $\Phi_e:S^p\To S^q$ by the formula $\Phi_e(f)=\Phi(fe)$. Of course, $\Phi_e$ is trivial. Indeed, taking $g=\Phi(e)$ we have
$$
\|\Phi_e(f)-fg\|_q= \|\Phi(fe)-f\Phi(e)\|_q\leq L(\Phi)\|f\|_B\|e\|_p\leq L(\Phi)\rk(e)^{1/p}\|f\|_p,
$$
where $\rk(e)$ is the dimension of the image of $e$.

\begin{lemma}\label{local} Let $\Phi:S^p_0\To S^q$ be a left centralizer, with $q$ finite. Then
$$\dist(\Phi,\mathscr M_B(S^p_0, S^q))=\sup_e\dist(\Phi_e,\mathscr M_B(S^p, S^q)),$$ where $e$ runs over all finite-rank projections in $B$.
\end{lemma}

\begin{proof}
It is clear that $\dist(\Phi,\mathscr M_B(S^p_0, S^q))\geq \dist(\Phi_e,\mathscr M_B(S^p, S^q))$ for every projection $e\in B$. % is obvious.
Let us prove the other inequality.
Let $\delta$ be a constant such that for every finite-rank projection $e$ there is a morphism $\phi_e$ so that
$$\|\Phi_e f- \phi_e(f)\|_q\leq \delta \|f\|_p\quad\quad(f\in S^p).$$

Let $\mathscr U$ be an ultrafilter refining the order filter on the set of finite-rank projections of $B$. We define a mapping $\phi:S^p_0\To S^q$ by the formula
\begin{equation}\label{lim}
\phi(f)=\lim_\mathscr U\phi_e(fe)
\end{equation}
where the limit is taken in the WOT. % {\sc wot}.
The definition makes sense because for each $f\in S^p_0$ one has $fe=f$ for sufficiently large $e$. For these projections we have $\|\Phi(f)- \phi_e(f)\|_q\leq \delta\|f\|_p$ and thus the net $(\phi_e(fe))_e$ is (essentially) bounded in $S^q$ and so in $B$. As bounded subsets of $B$ are relatively compact in the WOT we see that (\ref{lim}) defines a map from $S^p_0$ to $B$. But $\|\cdot\|_q$ is lower semicontinuous with respect to the restriction of the WOT to $S^q$ (see \cite[Corollary 2.3]{dilworth}) and so
$$
\|\Phi(f)- \phi(f)\|_q\leq \lim_\mathscr U \|\Phi(f)- \phi_e(f)\|_q\leq \delta\|f\|_p\quad\quad(f\in S^p_0).
$$
In particular $\phi(f)$ belongs to $S^q$. Finally that $\phi$ is a morphism follows from the fact that, for fixed $a\in B$, the map $b\mapsto ab$ is WOT-continuous on bounded sets of $B$.
\end{proof}

The sought-after result reads as follows.

\begin{theorem}\label{qlessp}
Let $0<q<p<\infty$. There is a constant $K=K(p,q)$ so that, for every left centralizer $\Phi:S^p_0\To S^q$, there is a morphism of left modules $\phi:S^p_0\To S^q$ satisfying
$$
\|\Phi(f)-\phi(f)\|_q\leq KL(\Phi)\|f\|_p \quad\quad(f\in S^p_0),
$$
where $L(\Phi)$ is the left centralizer constant of $\Phi$.
\end{theorem}

The proof combines a simple ultraproduct technique and some ``noncommutative gadgetry''.
Here we only recall some definitions, mainly for notational purposes.

Let $X$ be a quasi Banach space, $I$ an index set and $\mathscr U$ a countably incomplete ultrafilter on $I$. Let $\ell^\infty(I,X)$ be the space of bounded families of $X$ indexed by $I$ (furnished with the sup quasinorm) and
let $N_\mathscr U$ be the (closed) subspace of those $x\in\ell^\infty(I,X)$ such that $\|x_i\|_X\To 0$ along $\mathscr U$. The ultrapower of $X$ with respect to $\mathscr U$ is the quotient space $\ell^\infty(I,X)/N_\mathscr U$ with the quotient quasinorm. The class of the family $(x_i)$ in $X_\mathscr U$ is denoted by $[(x_i)]$.  Notice that if the quasinorm of $X$ is continuous one can compute the quasinorm in $X_\mathscr U$ by the formula
$
\|[(x_i)]\|=\lim_\mathscr U \|x_i\|_X
$.
Clearly, if $A$ is a Banach algebra, then so is $A_\mathscr U$ when equipped with the coordinatewise product $
[(a_i)][(b_i)]=[(a_ib_i)]$. If besides $X$ is a quasi Banach module over $A$, then the multiplication $[(a_i)][(x_i)]=[(a_ix_i)]$ makes $X_\mathscr U$ into a quasi Banach module over $A_\mathscr U$.

What we need to prove Theorem~\ref{qlessp} is the following.

\begin{lemma}\label{ultra}
Let $p,q,r\in (0,\infty)$ satisfy $q^{-1}=p^{-1}+r^{-1}$. If $\gamma:S^p_\mathscr U\To S^q_\mathscr U$ is a homomorphism of left modules over $B_\mathscr U$, then there is bounded family $(g_i)$ in $S^r$ such  that
$\gamma[(f_i)]=[(f_ig_i)]$ whenever $(f_i)$ is bounded in $S^p$.
\end{lemma}

\begin{proof}
This can be obtained as a combination of results by Raynaud, and Junge and Sherman. Let us explain how.

(1) There is a general construction, due to Haagerup, that associates to a given von Neumann algebra $\mathcal M$ the so-called (Haagerup, non-commutative) $L^p$ spaces $L^p(\mathcal M)$ for $0<p\leq\infty$. These spaces consist of certain (densely defined, closable, but in general discontinuous) operators acting on a common suitable Hilbert space which is related to $\mathcal M$ in a highly nontrivial way and $\mathcal M$ itself can be identified with $L^\infty(\mathcal M)$, as von Neumann algebras.
As it happens this provides the following generalization of H\"older inequality: suppose $p,q,r\in(0,\infty]$ are such that $q^{-1}=p^{-1}+r^{-1}$; if $f\in L^p(\mathcal M)$ and $g\in L^r(\mathcal M)$, then $fg\in L^q(\mathcal M)$ and
$\|fg\|_q\leq \|f\|_p\|g\|_r$, where the subscript indicates the quasinorm of the corresponding Haagerup space. Letting $p=\infty$ or $r=\infty$ one gets the module structures over $L^\infty(\mathcal M)$. See \cite{haa,pixu}.

(2) After that it is clear that that every $g\in L^r(\mathcal M)$ gives rise to a homomorphism (of left  $L^\infty(\mathcal M)$-modules) $\gamma:L^p(\mathcal M)\To L^q(\mathcal M)$ by multiplication: $\gamma(f)=fg$. Moreover, $\|\gamma:L^p(\mathcal M)\To L^q(\mathcal M)\|=\|g\|_r$.
Junge and Sherman proved in \cite[Theorem 2.5]{js} that all such homomorphisms arise in this way, which is crucial for us.

(3) The Haagerup spaces do not form any ``scale''. Indeed, by the very definition, one has $
L^p(\mathcal M)\cap L^q(\mathcal M)=0$ unless $p=q$. In particular, $L^p(B)$ (the Haagerup $L^p$ space corresponding to the choice $\mathcal M=B$) cannot be the same as `our' $S^p$. %The reason for this is that $\frak H_\mathcal B$ is far from being $\mathscr H$ or even $S^2$.
Nevertheless there is a system of isometric bimodule isomorphisms $\iota_p:S^p\To L^p(B)$ which are compatible with the product maps in the sense that
$\iota_q(fg)=\iota_p(f)\iota_r(g)$ whenever $f\in S^p$ and $g\in S^q$ with $
q^{-1}=p^{-1}+r^{-1}$.

The obvious consequence of this is that a map $u:S^p\To S^q$ is a homomorphisms of $B$-modules if and only if $\iota_q\circ u\circ \iota_p^{-1}:L^p(B)\To L^q(B)$ is a homomorphism of $L^\infty(B)$-modules.
Therefore replacing Schatten classes by Haagerup spaces and $B$ by $L^\infty(B)$ does not alter the Lemma.

 (4) Raynaud proved in \cite{raynaud} that given a von Neumann algebra $\mathcal M$ and a countably incomplete ultrafilter $\mathscr U$ one can represent the ultrapowers of the whole family of Haagerup spaces $L^p(\mathcal M)$ (for finite $p$) as the Haagerup spaces associated to some von Neumann algebra independent on $p$. Precisely: there is a von Neumann algebra $\mathcal N$ containing $L^\infty(\mathcal M)_\mathscr U$ and a system of surjective isometries $\kappa_p:L^p(\mathcal M)_\mathscr U\To L^p(\mathcal N)$ for $0<p<\infty$ compatible with the product maps in the following sense:  $p,q,r\in(0,\infty)$ are such that $q^{-1}=p^{-1}+r^{-1}$ and $(f_i)$ and $(g_i)$ are bounded families in $L^p(\mathcal M)$ and $L^r(\mathcal M)$, respectively, then
$$
(\kappa_p[(f_i)])(\kappa_r[(g_i)])= \kappa_q[(f_ig_i)],
$$
where the product in the left-hand side refers to spaces over $\mathcal N$ and those in the right-hand side to $\mathcal M$.

(5) Therefore we can regard $L^p(\mathcal M)_\mathscr U$ as a module over $\mathcal N$ and every homomorphism of $\mathcal N$-modules $\gamma:L^p(\mathcal M)_\mathscr U\To L^q(\mathcal M)_\mathscr U$ can be represented as $\gamma[(f_i)]=[(f_ig_i)]$, where $(g_i)$ is a bounded family
in $L^r(\mathcal M)$.

 (6) The proof of the Lemma will be complete if we show that every homomorphism of $L^\infty(\mathcal M)_\mathscr U$-modules $\gamma:L^p(\mathcal M)_\mathscr U\To L^q(\mathcal M)_\mathscr U$ is automatically a homomorphism of $\mathcal N$-modules. And this is so because on one hand $L^\infty(\mathcal M)_\mathscr U$ is dense in $\mathcal N$ in the strong operator topology induced by the (module) action on $L^2(\mathcal N)$ and, on the other hand, the restriction to bounded subsets of $\mathcal N$ of the strong operator topology induced by the action on $L^p(\mathcal N)$   does not depend on $0<p<\infty$ (see \cite[Lemma~2.3]{js}).
\end{proof}

%We are now ready for the

\begin{proof}[Proof of Theorem~\ref{qlessp}]
Assume on the contrary that there is a sequence of centralizers $\Phi_n:S^p_0\To S^q$ such that $L(\Phi_n)\leq 1$ and $\dist(\Phi_n, \mathscr M_B(S^p_0,S^q))\To\infty$.
In view of Lemma~\ref{local} we may assume that for each $n$ there is a finite-rank projection $e_n\in B$ such that $\Phi_n(f)=\Phi_n(fe_n)$ for all $f\in S^p_0$. Thus there is no loss of generality if we assume that each $\Phi_n$ is defined on the whole of $S^p$ and also that
$\dist(\Phi_n,\mathscr M_B(S^p,S^q))$ is finite for every $n$.

For each $n$ we take a morphism $\phi_n:S^p\To S^q$ such that $$
\delta_n=\dist(\Phi_n,\phi_n)\leq \dist(\Phi_n,
\mathscr M_B(S^p,S^q))+1/n. $$ Of course, $\delta_n\To\infty$ as
$n\To\infty$. Put $$ v_n=\frac{\Phi_n-\phi_n}{\delta_n}, $$ so that $v_n:S^p\To S^q$ is a bounded homogeneous mapping with
$\|v_n:S^p\To S^q\|\leq 1$ and $L(v_n)\leq L(\Phi_n)/\delta_n\To 0$ as
$n\To\infty$. By Lemma~\ref{lem:spare} we also have $Q(v_n)\To0$.% as $n$ increases.

Let $\mathscr U$ be a free ultrafilter on the integers and consider the
corresponding ultrapowers $ S^p_{\mathscr U}$ and $ S^q_{\mathscr U}$. We can use
the (probably nonlinear) maps $v_n$ to define
$v: S^p_{\mathscr U}\To S^q_{\mathscr U}$ by $$ v[(f_n)]=[(v_n(f_n))]. $$ Let us
check that $v$ is well defined. First, suppose $[(f_n)]=0$, that
is, $\|f_n\|_p\To 0$ along $\mathscr U$. As $\|v_n(f_n)\|_q\leq \|f_n\|_p$ we have $[(v_n(f_n))]=0$. Suppose now $[(f_n)]=[(g_n)]$. We must prove that
$[(v_n(f_n))]=[(v_n(g_n))]$. But
$$ \lim_\mathscr U\|v_n(f_n)-v_n(g_n)\|_q=
\lim_\mathscr U\|v_n(f_n)-v_n(g_n)-v_n(f_n-g_n)\|_q\leq \lim_\mathscr U
Q(v_n)\big{(}\|g_n\|_p+\|f_n-g_n\|_p\big{)}=0 $$ and the definition of $v$ makes
sense. Now it is nearly obvious that $v$ is a continuous
homomorphism of $B_\mathscr U$-modules. By Lemma~\ref{ultra} there is a bounded sequence $(u_n)$ in $
S^r$ representing  $v$ in the sense that $
v[(f_n)]=[(f_nu_n)]$
whenever $(f_n)$ is a bounded sequence in $S^p$, where $r^{-1}+p^{-1}=q^{-1}$.  This implies that $\dist(v_n,u_n)\To0$
along $\mathscr U$. In particular, for every $\e>0$, the set $ S=\{n\in\N:
0<\dist((\Phi_n-\phi_n)/\delta_n, u_n)<\e\} $ belongs to $\mathscr U$ and it
contains infinitely many indices $n$. For these $n$ we get $$
\dist(\Phi_n,\phi_n+\delta_nu_n)<\e\delta_n< 2\e\dist(\Phi_n,
\mathscr M_B(S^p,S^q)), $$ in striking contradiction with our choice of $\phi_n$.
\end{proof}

\begin{corollary}\label{coro}
If\: $0<q<p\leq\infty$, then $\Ext_B(S^p,S^q)=0$.
%\begin{itemize}
%\item $\Ext_B(S^p,S^q)=0$,
%\item $\Ext_B(S^p,\mathscr H)=0$.
%\end{itemize}
\end{corollary}

\begin{proof}
Corollary~\ref{all} and Theorem~\ref{qlessp}. For $p=\infty$ use Lemma~\ref{local}.
\end{proof}

\section{Isomorphisms of spaces of centralizers}\label{sec:isom}
Once we know that $\Ext_B(-,-)$ vanishes at certain couples $(S^s, S^r)$  we can use the functor $\Hom_B(-,-)$, fixing one of the arguments, to compare different spaces of extensions. This can be done either working directly with the extensions or using the corresponding centralizers. In the next two Sections we focus on centralizers; those readers acquainted with Yoneda's approach to $\Ext$ can see the linear counterpart in Section~\ref{sec:counterpart}. At the end of the day, the space $\Ext_B(S^p,S^q)$ depends only on the difference $q^{-1}-p^{-1}$, just as it happens to the space of homomorphisms; see \ref{eq:HomSpSq}.

\subsection{Two covariant transformations}\label{sec:cova}
Let us beging with the ``covariant'' case, corresponding to $\Hom(S^s,-)$.
%Let us take a look at the actions of $\Hom(-,S^r)$ and $\Hom(S^s,-)$ on centralizers.
To take advantage of the extra simplification provided by Lemma~\ref{mor}(b) we shall work with right centralizers. We have included a statement about bicentralizers, as well as the case $q_1<p_1$, so that we can use them in Section~\ref{sec:bicentralizers}.

So, given $p,q\in(0,\infty]$ we write $\mathscr C(S^p_0,S^p)_B$ for the space of right centralizers $\Phi: S^p_0\To S^q$. We denote by ${\mathscr C}(S^p_0,S^q)_B^\sim$ (respectively, $\mathscr C(S^p_0,S^p)_B^\approx$) the quotient by the subspace of trivial (respectively, bounded) centralizers.

\begin{proposition}\label{prop:cova1}Let $0<p_1,q_1<\infty$ and $s>q_1$.
%Let $0<p_1\leq q_1< s<\infty$.
We define $(p_2,q_2)$ by $p_2^{-1}+s^{-1}=p_1^{-1}$ and $q_2^{-1}+s^{-1}=q_1^{-1}$. Then, for each\: $\Phi\in \mathscr C(S^{p_1}_0,S^{q_1})_B$, there is  $\Phi^{(s)}\in \mathscr C(S^{p_2}_0,S^{q_2})_B$ such that
\begin{equation}\label{eq:cova1}
\|\Phi(gf)- (\Phi^{(s)} g)f\|_{q_1}\leq M\|g\|_{p_2}\|f\|_s\quad\quad(f,g\in\mathfrak F).
\end{equation}
Moreover:\begin{itemize}
\item[(a)] Any two centralizers satisfying the preceding estimate are strongly equivalent.
\item[(b)] If\: $\Phi$ is a bicentralizer, then so is\: $\Phi^{(s)}$.
\end{itemize}
\end{proposition}

\begin{proof}
Our choice of the parameters guarantees that $\Hom(S^s, S^{p_1})_B=S^{p_2}$ and $ \Hom(S^s, S^{p_1})_B=S^{p_2}$.

Let us first show that every homogeneous mapping $\Gamma:  S^{p_2}\To L(\mathscr H)$ satisfying an estimate
$$
\|\Phi(gf)- (\Gamma g)f\|_{q_1}\leq M\|g\|_{p_2}\|f\|_s\quad\quad(f\in\mathfrak F)
$$
is a right centralizer from $S^{p_2}$ to $S^{q_2}$. Pick $g\in S^{p_2}, a\in B$ and let us compare $\Gamma(ga)$ with  $(\Gamma g)a$. One has
$$%\begin{align*}
\|\Phi((ga)f)- (\Gamma (ga))f\|_{q_1}\leq M\|ga\|_{p_2}\|f\|_s\quad\text{and}\quad
\|(\Gamma g)af- \Phi(g(af)) \|_{q_1}\leq M\|g\|_{p_2}\|af\|_s,
$$%\end{align*}
hence
$$
\|(\Phi(ga)- (\Gamma g)a))f\|_{q_1}\leq M\|g\|_{p_2}\|a\|_B\|f\|_s\quad\quad(f\in\frak F),
$$
which is enough since $\|h\|_{q_2}=\sup\left\{\|gf\|_{q_1}: f\in \mathfrak F, \|f\|_s\leq 1\right\}$. A similar argument shows that any two centralizers fitting in  (\ref{eq:cova1}) are strongly equivalent, which gives (a).

Let us check (b) right now. Assuming that $\Phi$ is in addition a left centralizer we have
$$
\|(\Phi(agf)- a\Phi(gf)\|_{q_1}\leq M\|a\|_B\|gf\|_{p_1}.
$$
Also,
$$%\begin{align*}
\|\Phi(agf)- (\Gamma(ag))f\|_{q_1}\leq M\|ag\|_{p_2}\|f\|_{s}\quad\text{and}\quad
\|a\Phi(gf)- a(\Gamma g)f\|_{q_1}\leq M\|a\|_B\|g\|_{p_2}\|f\|_{s}.
$$%\end{align*}
Combining,
$$
\|\Gamma(ag)f- a(\Gamma g)f\|_{q_1}\leq M\|a\|_B\|g\|_{p_2}\|f\|_s,%\quad\quad(f\in\frak F),
$$
which yields (b).

The rest is straightforward from Theorem~\ref{qlessp}.
Take $\Phi\in \mathscr C(S^{p_1}_0,S^{q_1})_B$ and $g\in S^{p_2}$. The composition
$f\in S^s_0\longmapsto \Phi(gf)\in S^{q_1}$ is a right centralizer, with constant at most
 $R(\Phi) \|g\|_{p_2}$.
 As $q_1<s$, Theorem~\ref{qlessp} provides a linear map $\ell\in L(\mathscr H)$, depending on $g$, such that
 $\|\Phi(gf)-\ell\circ(f)\|_{p_2}\leq K \|g\|_{p_2}R(\Phi) \|f\|_{s}$, where $K=K(s,q_1)$. Selecting homogeneously such an $\ell$ gives the desired centralizer $\Phi^{(s)}$.
 
The restriction of  $\Phi^{(s)}$ to $\mathfrak F$ takes values in $S^{q_2}$: pick $g\in\frak F$ and let $e$ be the initial projection of $g$ so that $g=ge$. Then $\Phi^{(s)}g-(\Phi^{(s)}g)e$ belongs to $S^{q_2}$ and since $(\Phi^{(s)}g)e$ is continuous and has finite rank we see that $\Phi^{(s)}g\in S^{q_2}$.
\end{proof}

The following Proposition provides the ``inverse'' of the transformation defined in the preceding one. The notation aims to highlight this connection.

%Next we study the inverse transformation.% to be used in Section~\ref{qgp}.

\begin{proposition}\label{version}
Let $\Psi:S^{p_2}_0\To S^{q_2}_{\:}$ be a right centralizer, with $0<p_2, q_2\leq \infty$. Take $0<s < \infty$ and let $p_1$ and $q_1$ be given by $p_1^{-1}= p_2^{-1}+ s^{-1}$ and $q_1^{-1}= q_2^{-1}+ s^{-1}$. We define a mapping $\Psi_{(s)}:S^{p_1}_0\To S^{q_1}$ by
 $$\Psi_{(s)}(h)=\Psi(u|h|^{p_1/p_2})|h|^{p_1/s},
 $$
 where $u$ is the phase of $h$.
 Then $\Psi_{(s)}$ is a right centralizer.
Moreover:\begin{itemize}
\item[(a)] $\Psi_{(s)}$ is bounded if and only if $\Psi$ is bounded.
\item[(b)] If\: $\Psi$ is a bicentralizer, then so is\: $\Psi_{(s)}$.
\end{itemize}
\end{proposition}

\begin{proof}
Let us first prove that if $f_1g_1=f_2g_2$, then
\begin{equation}\label{account}
\| (\Psi f_1)g_1-(\Psi f_2) g_2	\|_{q_1}\leq M(\|f_1\|_{p_2}\|g_1\|_s+\|f_2\|_{p_2}\|g_2\|_s),
\end{equation}
for some $M$ independent on $f_i$ and $g_i$.
Taking adjoints in the proof of Lemma~\ref{lem:spare} one obtains $f\in S^{p_2}_0$, with $\|f\|_{p_2}\leq \Delta_{p_2/2}^{1/2}(\|f_1\|_{p_2}+\|f_2\|_{p_2})$, such that $f_i=fa_i$ for certain contractive $a_i\in\B$ whose final projections agree with the inicial projection of $f$.
 Now, since $f_1a_1g=f_2a_2g$ we have $a_1g_1=a_2g_2$.
For $i=1,2$, one has
$$
\|\Psi(f_i)g_i-\Psi(f)a_ig_i\|_{q_1}\leq C(\Psi)\|f\|_{p_2}\|g_i\|_s\leq M\left(\|f_1\|_{p_2} + \|f_2 \|_{p_2}\right)\|g_i\|_s
$$
and combining we arrive to
$$
\| (\Psi f_1)g_1-(\Psi f_2) g_2	\|_{q_1} \leq M(\|f_1\|_{p_2}+\|f_2\|_{p_2})(\|g_1\|_s+\|g_2\|_s).
$$
But $\Psi$ is homogeneous and since $f_1g_1=\alpha f_1 \alpha^{-1}g_1$ and  $f_2g_2=\beta f_2 \beta^{-1}g_2$, for $\alpha,\beta>0$, we also obtain
$$
\| (\Psi f_1)g_1-(\Psi f_2) g_2	\|_{q_1}\leq M(\alpha\|f_1\|_{p_2}+ \beta \|f_2\|_{p_2})(\alpha^{-1}\|g_1\|_s+\beta^{-1}\|g_2\|_s).
$$
Minimizing the right-hand side over $\alpha, \beta>0$ we obtain
$$%\begin{align*}
\| (\Psi f_1)g_1-(\Psi f_2) g_2	\|_{q_1} \leq M
\left(\|f_1\|_{p_2}^{1/2}\|g_1\|_s^{1/2}+\|f_2\|_{p_2}^{1/2}\|g_2\|_s^{1/2}\right)^2
\leq 2M (\|f_1\|_{p_2}\|g_1\|_s+\|f_2\|_{p_2}\|g_2\|_s),
$$%\end{align*}
which proves (\ref{account}).

We now prove that $\Psi_{(s)}\in\mathscr C(S^{p_1}_0,S^{q_1})_B$. Take $h\in\mathfrak F$ and $a\in B$. Let $v$ be the phase of $ha$ so that $ha=v|ha|^{p_1/p_2} |ha|^{p_1/s}$. One has
$$
(\Psi_{(s)}h)a = \Psi(u|h|^{p_1/p_2})|h|^{p_1/s} a
, \quad \text{while}\quad \Phi_{(s)}(ha)=
\Psi(v|ha|^{p_1/p_2})|ha|^{p_1/s}.
$$
And since 
$u|h|^{p_1/p_2} |h|^{p_1/s} a=v|ha|^{p_1/p_2})|ha|^{p_1/s}= ha$
we may apply (\ref{account}) to get
\begin{align*}
\|\Phi_{(s)}(ha)- (\Phi_{(s)}h)a	\|_{q_1}&\leq 
M \big{(}\big{\|}u|h|^{p_1/p_2}\big{\|}_{p_2}  \big{\|} |h|^{p_1/s} a \big{\|}_s+ 
\big{\|} v|ha|^{p_1/p_2} \big{\|}_{p_2} \big{\|} |ha|^{p_1/s} \big{\|}_s\big{)}
\\
&\leq  M\big{(}\|h\|_{p_1}^{p_1/p_2} \|h\|_{p_1}^{p_1/s}\|a\|_B + \|ha\|_{p_1}^{p_1/p_2} \|ha\|_{p_1}^{p_1/s}\big{)}
\leq  2M\|h\|_{p_1}\|a\|_B.%\infty + \|ha\|_{p_1}).
\end{align*}
%The estimate (\ref{phir}) is clear from (\ref{account}). Finally, if $\Gamma:L_q\to L^0(\mathscr R,\tau)$ is any mapping such that $\|\Gamma(fg)-f\Phi g\|_q\leq C\|f\|_r\|g\|_p$, then given $h\in L_q$ we may take $f=u|h|^{q/r}$ and $g=|h|^{q/p}$ to obtain
%$$
%\|\Gamma(h)-\Phi^{(r)}(h)\|_q=\|\Gamma(fg)-f\Phi g\|_q\leq C\|f\|_r\|g\|_p=C\|h\|_q.
%$$
Finally, we prove the ``moreover'' part. (a) is obvious. As for (b),
take $a\in B$ and $h\in S^{p_1}$. Let $w$ be the phase of $ah$ so that $ah=w|ah|^{p_1/p_2} |ah|^{p_1/s}$. One has
$$
\Psi_{(s)}(ah) = \Psi(w|ah|^{p_1/p_2})|ah|^{p_1/s}  \quad \text{and}\quad a\Psi_{(s)}(h)=
a\Psi(u|h|^{p_1/p_2})|h|^{p_1/s}.
$$
Assuming that $\Psi$ is a bicentralizer,
$$
\| a\Psi(u|h|^{p_1/p_2})- \Psi(au|h|^{p_1/p_2})\|_{q_2}\leq L(\Psi)\|a\|_B\big{\|} \:|h|^{p_1/p_2}\: \big{\|}_{p_2}.
$$
Now, since $w|ah|^{p_1/p_2}|ah|^{p_1/s}=au|h|^{p_1/p_2}|h|^{p_1/s}=ah$,
\begin{align*}
\big{\|} \Psi_{(s)}&(ah) - a\Psi_{(s)}(h)\big{\|}_{q_1}=
\left\| \Psi(w|ah|^\frac{p_1}{p_2})|ah|^\frac{p_1}{s}  - a\Psi(u|h|^\frac{p_1}{p_2})|h|^\frac{p_1}{s}\right\|_{q_1}\\
&\leq 
\left\| \Psi(w|ah|^\frac{p_1}{p_2})|ah|^\frac{p_1}{s}  - \Psi(au|h|^\frac{p_1}{p_2})|h|^\frac{p_1}{s}\right\|_{q_1}+
\left\| \Psi(au|h|^\frac{p_1}{p_2})|h|^\frac{p_1}{s}  - a\Psi(u|h|^\frac{p_1}{p_2})|h|^\frac{p_1}{s}\right\|_{q_1}\\
& \leq M\left( \|w|ah|^\frac{p_1}{p_2}\|_{p_2}\||ah|^\frac{p_1}{s}\|_s + \| au|h|^\frac{p_1}{p_2}\|_{p_2}\|\:|h|^\frac{p_1}{s}\:\|_s \right)+
L(\Psi)\|a\|_B \left\| |h|^\frac{p_1}{p_2}  \right\|_{p_2} \|\:|h|^\frac{p_1}{s}\|_s\\
&\leq M \|a\|_B  \|\:|h|\|_{p_1}^{{p_1/p_2}+ {p_1/s}}\\
&\leq M\|a\|_B\| h\|_{p_1}.
\end{align*}
This completes the proof.
\end{proof}

In contrast to Proposition~\ref{prop:cova1}, which is supported on Theorem~\ref{qlessp}, the preceding result is  ``elementary'' as it only depends on the fact that H\"older inequality is sharp in the the following sense: if $q^{-1}=p^{-1}+s^{-1}$, then every $h\in S^q$ can be factorized as $h=fg$, with $\|h\|_q=\|f\|_p\|g\|_s$.
 
\medskip

%\begin{remark}
By following the proofs of Proposition~\ref{prop:cova1} and \ref{version} one can obtain a curious ``extension'' result: every right centralizer $\Phi_0: S^{p}_0\To S^{q}$ admits an extension $\Phi: S^{p}\To L(\Hi)$ which is a right centralizer from $S^p$ to $S^q$.
%\end{remark} 

\subsection{A contravariant transformation}\label{sec:contra}
We turn to the action of $\Hom(-,S^r)$.
Please be careful with the positions of the indices.

\begin{proposition}\label{prop:contra}
Assume $p_1^{-1}+p_2^{-1}= q_1^{-1}+q_2^{-1}=r^{-1}$, with $0<r<\infty$.
%Assume $0<r<p_1\leq q_1\leq\infty$. Define new indices $p_2$ and $q_2$ 
Then to each right centralizer $\Phi: S^{p_1}_0\To S^{q_1}$ there corresponds a mapping  $\Gamma:  S^{q_2}\To L(\Hi)$ which is a  left centralizer from $S^{q_2}$ to $ S^{p_2}$ and obeys the estimate
\begin{equation}\label{eq:contra}
\|g(\Phi f) +(\Gamma g) f \|_r\leq M \|g\|_{q_2} \|f\|_{p_1},\quad\quad(f\in \frak F ,g\in S^{q_2}).
\end{equation} 
Such a $\Gamma$ is unique, up to strong equivalence.
%This correspondence establishes a linear isomorphism between the spaces of (classes of) centralizers $\mathscr C(S^{p_1}_0,S^{q_1})^\approx_B$ and  $C_B(S^{q_2}_0,S^{p_2})^\approx$ which induces an isomorphism between 
%$\mathscr C(S^{p_1}_0,S^{q_1})^\sim_B$ and  $C_B(S^{q_2}_0,S^{p_2})^\sim$.XXX
\end{proposition}

\begin{proof}
It is clear that any homogeneous mapping $\Gamma: S^{q_2}\To L(\Hi)$ fulfilling (\ref{eq:contra}) is a left centralizer from $S^{q_2}$ to $ S^{p_2}$.
% Pick $g\in S^{q_2}, a\in B$. Then 
To obtain one, take $g\in S^{q_2}$ and consider the composition $f\in S^{p_1}_0\longmapsto g(\Phi f)\in S^r$. This is a right centralizer, with centralizer constant at most $\|g\|_{q_2}R(\Phi)$ and since $r< p_1$ Theorem~\ref{qlessp} provides us with a morphism of right modules $\gamma_g: \mathfrak F\To B$ such that 
$\|g(\Phi f)+   \gamma_g(f)\|_r\leq K(p_1,r)R(\Phi)\|g\|_{q_2}\|f\|_{p_1}$. Selecting a linear map in $L(\Hi)$ implementing $\gamma_g$ in a homogeneous manner gives $\Gamma$.
\end{proof}

\subsection{The algebra behind all this}\label{sec:counterpart}
In this Section we briefly comment on the algebraic counterpart of Sections~\ref{sec:cova} and \ref{sec:contra}. We assume that the reader is acquainted with Yoneda Ext groups, as presented in \cite{hs}. The remainder of the paper is independent on this Section.
\medskip

Let us begin with the covariant case, considering right $B$-modules. Suppose we are given an extension of quasi Banach modules
\begin{equation}\label{we}
%\begin{equation}\label{s}
\begin{CD}
 0@>>> Y@>\imath >> Z@>\pi>> X@>>> 0
 \end{CD}
%\end{equation}
\end{equation}
If $E$ is another right module and we apply $\Hom(E,-)_B$ to (\ref{we}) we get an exact sequence (of linear spaces)
\begin{equation}\label{hom-ext}
\begin{CD}
0 @>>> \Hom(E,Y)_B @>{\imath_\circ}>> \Hom(E,Z)_B @>{\pi_\circ}>> \Hom(E,X)_B @>\alpha>> \Ext(E,Y)_B
\end{CD}
\end{equation}
Notice that $\imath_\circ$  and  $\pi_\circ$ are just the functorial images of $\imath$ and $\pi$. %, and similarly with $\pi_\circ$.
The connecting map $\alpha$ sends each homomorphism $\phi$ into the (class of the) lower extension in the pull-back diagram
$$
\begin{CD}
0@>>> Y@> \imath >> X @>\pi>> Z@>>> 0\\
&& @|  @AAA @AA\phi A \\
0@>>> Y@>>> \PB @>>> E@>>> 0\\
\end{CD}
$$
If $\Ext(E,Y)_B=0$, then (\ref{hom-ext}) represents an extension of $\Hom(E,Z)_B$ by $\Hom(E,Y)_B$. If, besides, $E$ is a bimodule, then (\ref{hom-ext}) is an extension of right modules.% All this can be seen in \cite{cc}.
\medskip

 Proposition~\ref{prop:cova1} corresponds to the case where $X=S^{p_1}, Y=S^{q_1}, Z$ is the extension induced by a right centralizer $\Phi:S^{p_1}_0\To S^{q_1}$ and $E=S^s$, with $s>q_1$. Then Theorem~\ref{qlessp} states that $\Ext(E,Y)_B=0$ and $\Hom(S^s, S^{p_1})_B= S^{p_2}, \Hom(S^s, S^{q_1})_B= S^{q_2}$ with $p^{-1}_1=s^{-1}+p^{-1}_2,
 q^{-1}_1=s^{-1}+q^{-1}_2$, so that (\ref{hom-ext}) can be seen as an extension of $S^{p_2}$ by $S^{q_2}$ and, actually, one has $\Hom(S^s, Z_\Phi)_B\cong Z_{\Phi^{(s)}}$.

 \medskip

 % S where $\Phi\in \mathscr C(S^{p_1}_0,S^{q_1})_B$. 
In a similar vein, if we apply $\Hom(-,F)_B$ to (\ref{we}) we obtain the exact sequence
\begin{equation}\label{hom-ext-con}
\begin{CD}
0 @>>> \Hom(X,F)_B @>{\pi^\circ}>> \Hom(Z,F)_B @>{\imath^\circ}>> \Hom(Y,F)_B @>\beta>> \Ext(X,F)_B
\end{CD}
\end{equation}
Here, $\beta$ sends a given homomorphism $\phi:Y\To E$ into the (class of the) lower row of the push-out diagram
$$
\begin{CD}
0@>>> Y@>\imath >> X @>>> Z@>>> 0\\
&& @V\phi VV @VVV @| \\
0@>>> E@>>> \PO @>>> Z@>>> 0\\
\end{CD}
$$
If $\Ext(X,F)_B=0$, then (\ref{hom-ext-con}) is an extension of $\Hom(Y,F)_B$ by $\Hom(X,E)_B$ which lives in the category of left modules if $E$ is a bimodule.

\medskip

Proposition~\ref{prop:contra} corresponds to the case where (\ref{we}) is the extension induced by a right centralizer $\Phi:S^{p_1}_0\To S^{q_1}$ and $F=S^r$, with $0<r<p_1$. Theorem~\ref{qlessp} states that $\Ext(X,F)_B=0$ and since
$\Hom(Y,F)_B=S^{q_2},  \Hom(X,F)_B=S^{p_2}$, then (\ref{hom-ext-con}) is an extension of left modules, with quotient $S^{q_2}$ and subspace $S^{p_2}$. 
Actually, if $\Phi$ and $\Gamma$ are as in (\ref{eq:contra}), one has $\Hom(Z_\Phi, S^r)_B\cong Z_{\Gamma}$, as left modules.
 \medskip

The procedure described in Proposition~\ref{version} works as a tensor product. And indeed it is. It can be proved that if $Z_\Psi$ is the completion of $S^{q_2}\oplus_\Psi S^{p_2}_0$, then $Z_{\Psi_{(s)}}$ represents the tensor product of $X_\Psi$ and $S^s$ in the category of quasi Banach $B$-modules. (Here, we consider $X_\Psi$ as a right module and $S^s$ as a left module: the resulting object is a right module because $S^s$ is a bimodule.)  This means that the bilinear operator $\theta:Z_\Psi\times S^s\To X_{\Psi_{(s)}}$ defined by $\theta((g,f),h)=(gh,fh)$ has the following universal property: for every quasi Banach space $V$ and every bilinear operator $\beta:X_\Psi\times S^s\To V$ which is balanced in the sense of satisfying the identity $\beta(xa,h)=\beta(x,ah)$ for $a\in B, x\in X_\Psi, h\in S^s$, there is a unique linear operator $\lambda:X_{\Psi^{(s)}}\To V$ such that $\lambda(\theta(x,h))=\beta(x,h)$.

This can be obtained combining Pavlov \cite{pavlov} with the ideas of \cite{tensor}. We will not insist on this point.

\section{The case $p\leq q$}\label{qgp}

\subsection{Twisted Hilbert spaces}
In this Section we describe the extensions of $S^p$ by $S^q$, with $0<p\leq q\leq\infty$, by means of the so-called
twisted Hilbert spaces. These are self-extensions of $\mathscr H$ in the category of (quasi) Banach spaces, that is, short exact sequences of (quasi) Banach spaces and operators
\begin{equation}\label{dia:THS}
\begin{CD}
0@>>> \Hi @>\jmath>> T @>\varpi >> \Hi @>>> 0
\end{CD}
\end{equation}
As a matter of fact, the middle space $T$ must be (isomorphic to) a Banach space \cite[Theorems 4.3(iii) and 4.10]{k} and has type $2-\e$ and cotype $2+\e$ for every $\e>0$; see \cite[Corollary~1]{elp} or \cite[Theorems 6.4 and 6.5]{k-type} for the  ``type part'' and then \cite[Proposition 11.10]{DJT} for the ``cotype part''. Moreover, $T$ is itself isomorphic to a Hilbert space if and only if (\ref{dia:THS}) splits. The existence of nontrivial twisted Hilbert spaces was first established by Enflo, Lindenstrauss, and Pisier \cite{elp}. Later on Kalton and Peck \cite{kp} constructed fairly concrete examples, among them the nowadays famous Kalton-Peck space $Z_2$.

As it is well-known, twisted Hilbert
spaces are in correspondence with quasilinear maps on $\mathscr H$, that is, homogeneous maps $\phi:\mathscr H\To\mathscr H$ satisfying an estimate of the form
$$
\|\phi(x+y)-\phi(x)-\phi(y)\|_\mathscr H\leq Q(\|x\|_\mathscr H + \|x\|_\mathscr H)\quad\quad(x,y\in\mathscr H).
$$
(As we did in Section~\ref{they} we can replace the target space by a larger ambient space, or consider $\phi$ defined only on some dense subspace, or both. However, as linear spaces are free modules over the ground field, this is unnecessary to elaborate the theory.) All this can be seen in \cite{bl, cg, k-handbook, kal-mon}. Incidentally, the space $Z_2$ just mentioned is the space one obtains letting $p=2$ and $\varphi(s,t)=s$ in (\ref{eq:phi}).

\subsection{Extensions of $S^2$ by $K$}\label{sec:SpK}
Let us explain how twisted Hilbert spaces give rise to module extensions of the Schatten classes. Consider an exact sequence as in (\ref{dia:THS}).  Without loss of generality we can assume that $\jmath$ is an isometry onto $\ker\varpi$. Also, we may fix a constant $C$ once and for all so that for every $y\in \Hi$ there is $z\in T$ such that $\|z\|_T\leq C\|y\|_{\Hi}$ and $h=\varpi(z)$

Now, suppose $u\in S^2$. Then $u$ factors through $\ell_1$ and so it lifts to $T$ in the sense that there is a bounded $\widehat{u}: \Hi\To T$ such that $u=\varpi\circ \widehat{u}$. Actually, if $\sum_n s_nx_n\otimes y_n$ is a Schmidt expansion of $u$ we may take $\widehat{u}= \sum_n s_nx_n\otimes z_n$, just selecting $z_n\in T$ such that $y_n=\varpi(z_n)$, with $\|z_n\|\leq C$, where $C$ is as before. Note that for $h\in\Hi$ one has $\widehat{u}(h)=\sum_n s_n\langle h|x_n\rangle z_n$, hence 
$\|\widehat{u}(h)\|_T\leq C|(s_n)|_2\| h\|_{\Hi}$ and $\widehat{u}$ is compact, with $\|\widehat{u}:\Hi\To T\|\leq C\|u\|_2$.

This lifting property allows us to construct an extension of $S^2$ by $K$ as follows. We set $X=\{x\in K(\mathscr H,T): \varpi\circ x\in S^2\}$ quasinormed by
$\|x\|=\max(\|x: \mathscr H\To T\|, \|\varpi\circ x\|_2)$. Then $X$ is a quasinormed right $B$-module under the product $xa=x\circ a$, where $a\in B$. We have an exact sequence of homomorphisms
\begin{equation}\label{KSpp}
\begin{CD}
0@>>> K @>\jmath_\circ >> X @>\varpi_\circ>> S^2 @>>> 0.
\end{CD}
\end{equation}
Clearly, $\imath_\circ$ is an isometry, and $\pi_\circ$ is onto and open. This implies that $X$ is complete and so (\ref{KSpp}) is an extension of quasi Banach modules. Some comments are in order:

\begin{itemize}
\item While one can replace $2$ by any $p\in(0,2)$ in the precedings considerations to obtain an extension of $S^p$ by $K$ in the category of quasi Banach right $B$-modules, the proof breaks down for $p>2$.

\item Of course it remains the question about the splitting of the sequence (\ref{KSpp}). It turns out that (\ref{KSpp}) splits as an extension of right Banach $B$-modules if and only if (\ref{dia:THS}) splits as an extension of Banach spaces: the ``if part'' is clear since if $v:\Hi\To T$ is a right inverse of $\varpi$, then $v_\circ: S^2\To X$ is a right inverse for $\varpi_\circ$. As for the converse, suppose $\alpha:S^2\To X$ is a homomorphism such that $(\varpi_\circ)\alpha$ is the identity on $S^2$. Let us fix a normalized $x\in\Hi$. Now, for every $y\in\Hi$ one has
$$
\alpha(x\otimes y)= \alpha\big{(}(x\otimes y)(x\otimes x)\big{)}=
\big{(}\alpha(x\otimes y)\big{)}(x\otimes x).
$$
It follows that $
\alpha(x\otimes y)=x\otimes a(y)$, for some linear map $a:\Hi\To T$, which is necessarily a bounded section of $\varpi$.% XXX (since $\|x\otimes y\|_2=\|y\|_\Hi$ and )
\end{itemize}

\subsection{Extensions of $S^p$ by $S^2$, with $p<2$, via $\gamma$-summing operators}
In this Section we begin to fill the gap between the indices of the submodule $S^q$ and the quotient module $S^p$. The idea is to fix $q=2$ and study the $\gamma$-summing norm of the ``obvious'' lifting of operators in $S^p$ with $0<p<2$.

We require some basic facts from the theory of absolutely summing operators that the reader can consult in \cite[Chapter 12]{DJT}.
The key notion is the following: an operator $v: E\To F$ acting between Banach spaces is $\gamma$-summing if there is a constant $c$ such that for every finite sequence $(x_k)_{1\leq k\leq n}$ in $E$ one has
\begin{equation}\label{eq:as}
\left(\int_\mathscr S \Big{\|} \sum_{1\leq k\leq n} g_k(s)v(x_k)\Big{\|}^2 dP(s) \right)^{1/2}\leq c\cdot \sup_{x'}\left(\sum_{1\leq k\leq n}|\langle x', x_k\rangle|^2\right)^{1/2},
\end{equation}
where the sup is taken for $x'$ in the unit ball of $E'$ and $(g_k)$ is a sequence of independent standard Gaussian variables on a probability space $(\mathscr S, P)$.
The least constant $c$ for which the preceding inequality holds is called the $\gamma$-summing norm of $v$ and will be denoted by $\|v\|_\gamma$.

\begin{lemma}\label{lem:u}
Let (\ref{dia:THS}) be a twisted Hilbert space and let $u$ be a finite rank operator on $\Hi$. Let $\sum_{1\leq k\leq n} s_kx_k\otimes y_k$ be a Schmidt expansion of $u$ and take $z_k\in T$ such that $\varpi(z_k)=y_k$ with $\|z_k\|\leq C$. Define $\widehat{u}:\Hi\To T$ by $\widehat{u}=\sum_{1\leq k\leq n} s_kx_k\otimes z_k$.

Then, for every $0<p<2$, there is a constant $L=L(p,T)$ such that $\|\widehat{u}\|_\gamma\leq L\|u\|_p$.
\end{lemma}

\begin{proof}
Let $E=\spa(x_1,\dots,x_k)$ be the ``initial subspace'' of $u$ and let $e$ be the orthogonal projection of $\Hi$ onto $E$, so that $u=u_0\circ e$, with $u_0=u|_E$. We consider the operator $\widehat{u}_0:E\To T$ given by $\widehat{u}_0=\sum_{1\leq k\leq n} s_kx_k\otimes z_k$. The sequence $(x_k)$ is an orthonormal basis of $E$ and so (cf. \cite[Theorem 12.15]{DJT})
\begin{align*}
\| \widehat{u}_0 \|_\gamma&= \left(\int_\mathscr S \Big{\|} \sum_{1\leq k\leq n} g_k(s)\widehat{u}_0(x_k)\Big{\|}^2 dP(s) \right)^{1/2}\\
&=
\left(\int_\mathscr S \Big{\|} \sum_{1\leq k\leq n} g_k(s)s_k z_k\Big{\|}^2 dP(s) \right)^{1/2} \leq T^\gamma_p\left( \sum_{k}C^p s_k^p	\right)^{1/p}= T^\gamma_p\cdot C\cdot \|u\|_p,
\end{align*}%$$ 
where $T^\gamma_p$ is the Gaussian type $p$ constant of $T$.
\end{proof}

Now, for each $u\in K$, we fix a Schmidt expansion which will be called the ``prescribed'' expansion of $u$.
It is assumed that this choice is homogeneous in the sense that if $\sum_n s_nx_n\otimes y_n$ is the prescribed expansion of $u$ and $\lambda$ is a complex number with polar decomposition $\lambda=\sigma|\lambda|$, then the prescribed expansion of $\lambda u$ is $\sum_n |\lambda|s_n (\sigma x_n\otimes y_n)$.

\begin{corollary}\label{cor:tildephi}
Let $\phi:\Hi\To\Hi$ be a quasilinear map. We define a mapping on $\frak F$ taking $\tilde{\phi}(u)=\sum_k s_k x_k\otimes \phi(y_k)$, where $\sum_ks_kx_k\otimes y_k$ is the prescribed expansion of $u$. Then $\tilde \phi: S^p_0\To S^q$ is a right centralizer provided $0<p<2$ and $q>p$.

Moreover, $\tilde{\phi}$ is essentially independent on the prescribed expansion in the sense that any other choice would lead to a centralizer strongly equivalent to $\tilde{\phi}$.
\end{corollary}

\begin{proof}Consider the twisted Hilbert space induced by $\phi$
$$
\begin{CD}
0@>>> \Hi @>\jmath>> \Hi\oplus_\phi\Hi @>\varpi >> \Hi @>>> 0.
\end{CD}
$$
The fact that the quasinorm of $T=\Hi\oplus_\phi\Hi$ is only equivalent to a norm will not cause any harm to the ensuing argument.
First of all note that for every $y$ in $\Hi$ one has $\|(\phi(y),y)\|_\phi=\|y\|$ and $\varpi((\phi(y),y))=y$. We define two ``lifting'' maps $B,\Lambda: \frak F\To \Pi_\gamma(\Hi, T)$ by the formulae
\begin{equation}
B(u)=\sum_n s_nx_n\otimes(\phi(y_n),y_n)\quad\quad\text{and}\quad\quad \Lambda(u)=\sum_n s_nx_n\otimes(0,y_n),
\end{equation}
where $\sum_ks_kx_k\otimes y_k$ is the prescribed expansion of $u$.
According to Lemma~\ref{lem:u} $B$ is a bounded map from $S^p_0$ to $ \Pi_\gamma(\Hi, T)$, while $\Lambda$ is a morphism of right-modules since $\Lambda(u)(h)=(0,u(h))$ for $h\in\Hi$ and so $\Lambda(ua)= \Lambda(u)a$ for $a\in B$. Incidentally, this shows that $\Lambda$ does not depend on the prescribed expansion of $u$.

Being $B(u)$ and $\Lambda(u)$ liftings of $u$, the difference $B(u)-\Lambda(u)$ can be interpreted as an operator on $\Hi$: actually we have $B(u)-\Lambda(u)=\jmath\circ\tilde{\phi}(u)$. This implies that $\tilde{\phi}:S^p_0\To  \Pi_\gamma(\Hi)$ is a right-centralizer:
$$
\|\tilde{\phi}(ua)-\tilde{\phi}(u)a\|_\gamma=
\|B(ua)-B(u)a\|_\gamma\leq M \big{(} \|ua\|_p+\|u\|_p\|a\|_B \big{)}\quad\quad(u\in S^p_0, a\in B). 
$$
This completes the proof for $q\geq 2$ since $\Pi_\gamma(\Hi)=S^2$, with (``universally'') equivalent norms. The ``uniqueness'' part is clear: $\Lambda(u)$ depends only on $u$ and if $\sum_ns_nx_n'\otimes y_n'$ is another Schmidt expansion of $u$, then $\big{\|}\sum_ns_nx_n'\otimes (\phi(y_n'),y_n')-B(u)\big{\|}_\gamma\leq M\|u\|_p$, where $M$ is a constant depending only on $p$ and $Q[\phi]$ -- through the modulus of concavity of  $ \Hi\oplus_\phi\Hi$.

\medskip

Next we prove that the map $\tilde\phi$ is still a right-centralizer when regarded as a map from  $S_0^p$ to $S^q$ with $0<p<q<2$.

Take $s>0$ so that $q^{-1}=2^{-1}+s^{-1}$ and then $p_2<2$ so that $p^{-1}=p_2^{-1}+s^{-1}$. We know that 
 $\tilde\phi:S^{p_2}_0\To S^2$ is a centralizer.
 
We introduce a second choice of the Schmidt expansions on $S^{p_2}$ as follows.
For every normalized $f\in S^{p_2}$ there is a unique normalized $u\in S^p$ such that 
$f=v|u|^{p/p_2}$, where $v$ is the phase of $u$. Now, if $\sum_ns_nx_n\otimes y_n$ is the prescribed expansion of $u$, then $\sum_ns_n^{p/p_2}x_n\otimes y_n$ is a Schmidt expansion of $f$ and the map $\Psi: S^{p_2}_0\To S^2$ defined by
$$
\Psi(f)=\sum_ns_n^{p/p_2}x_n\otimes \phi(y_n)
$$
is a centralizer -- it is strongly equivalent to  $\tilde\phi$.

Let us activate Proposition~\ref{version} to conclude that if $u=v|u|$ is the polar decomposition of $u\in S^p_0$, then
 the formula
 $$
\Psi_{(s)}(u)=\Psi(v|u|^{p/p_2}) |u|^{p/s}
$$
defines a centralizer from $S^p_0$ to $S^q$. But $
\Psi_{(s)}$ agrees with our old friend $\tilde\phi$. Indeed, if $\sum_ns_nx_n\otimes y_n$ is the prescribed expansion of $u$, then $v=\sum_nx_n\otimes y_n$ and $|u|= \sum_ns_nx_n\otimes x_n$, etc, and so
$$
\Psi_{(s)}u=\Psi(v|u|^{p/p_2}) |u|^{p/s}
= \Psi\left(\sum_n s_n^{p/p_2} x_n\otimes y_n \right)\! \left(\sum_n s_n^{p/s} x_n\otimes x_n \right)= \sum_ns_n x_n\otimes \phi(y_n) =\tilde{\phi}u,
$$
and we are done.
\end{proof}

\subsection{Self-extensions of $S^2$ via Pisier's lifting}\label{sec:S2Pisier}
Consider again a twisted Hilbert space, as in (\ref{dia:THS}). Applying $\Pi_\gamma(\Hi,-)$ we obtain the ``incomplete'' exact sequence
$$
\begin{CD}
0@>>> \Pi_\gamma(\Hi) @>\jmath_\circ >>\Pi_\gamma(\Hi, T) @>\varpi_\circ >> \Pi_\gamma(\Hi)
\end{CD}
$$%\end{equation}
Note that $\Pi_\gamma(\Hi)=S^2$, with (``universally'') equivalent norms. Let us show that the preceding sequence is actually a self-extension of $S^2$:

\begin{proposition}\label{prop:Pis}
With the preceding notations the map $\varpi_\circ$ is surjective and so
\begin{equation}\label{dia:Pi}
\begin{CD}
0@>>> S^2 @>\jmath_\circ >>\Pi_\gamma(\Hi, T) @>\varpi_\circ >> S^2 @>>> 0
\end{CD}
\end{equation}
is a self-extension of $S^2$ in the category of right Banach $B$-modules. 
\end{proposition}

\begin{proof}
It suffices to see that $\varpi_\circ$ is ``almost open''. Let $u$ be a finite-rank operator and let $\sum_ks_kx_k\otimes y_k$ be a Schmidt expansion of $u$. As $\Hi$ is $B$-convex we can apply Pisier's lifting in \cite[Theorem and Final Remark]{P} to obtain a finite sequence $(z_k)$ in $T$ such that $\varpi(z_k)=s_ky_k$ and
$$
\left(\int_\mathscr S \Big{\|} \sum_{1\leq k\leq n} g_k(s)z_k\Big{\|}^2 dP(s) \right)^{1/2}\leq M
\left(\int_\mathscr S \Big{\|} \sum_{1\leq k\leq n} g_k(s)s_ky_k\Big{\|}^2 dP(s) \right)^{1/2},
$$
where the $g_k$'s are as in (\ref{eq:as}) and  $M$ depends only on $T$.
Set $\widehat{u}=\sum_{1\leq k\leq n} x_k\otimes z_k$. Then $\widehat{u}$ is a lifting of $u$ and the preceding inequality shows that $\|\widehat{u}\|_\gamma\leq M_1\|u\|_2$, where $M_1$ is a constant depending only on $T$.
\end{proof}

Since $\Pi_\gamma(\Hi, T)$ contains a copy of $T$ (think of the rank-one operators) 
it is clear that (\ref{dia:Pi}) splits as an extension of Banach spaces (or as one of Banach modules) if and only if so (\ref{dia:THS}) does, which happens if and only if $T$ is a Hilbert space.

\subsection{The spatial part of a centralizer}
In the preceding Sections we have seen that quasilinear maps on $\Hi$ (equivalently, twisted Hilbert spaces) induce right centralizers on the Schatten classes (equivalently, right module extensions): Corollary~\ref{cor:tildephi} is particularly clear in this respect. The next result shows that, conversely, every centralizer gives rise to a quasilinear map that can be properly called its ``spatial part''.

\begin{lemma}\label{lem:spatial} Let $0<p\leq q\leq \infty$.
To each right centralizer $\Phi: S^p_0\To S^q$ there corresponds a quasilinear map $\phi:\Hi\To\Hi$ such that
\begin{equation}\label{times}
\|\Phi(x\otimes y)-  x\otimes \phi(y)\|_q\leq M\|x\|\|y\|%\quad\quad(x,y\in\mathscr H).
\end{equation}
for some constant $M$ and all $x,y\in\Hi$. 
Moreover:
\begin{itemize}
\item[(a)] Such a $\phi$ is unique, up to strong equivalence.
%\item If $\Phi$ is a trivial centralizer, then $\phi$ is trivial, as a quasilinear map.
\item[(b)] The map $J:\mathscr C(S^p_0,S^q)^\sim_B\To \mathscr Q(\Hi)^\sim$ defined by declaring $[\phi]=J[\Phi]$ if (\ref{times}) holds is correctly defined and linear.%. It is pretty obvious that $J$ is linear and it is even a morphism of two-sided modules when XXX
\end{itemize}
\end{lemma}

\begin{proof}
%We first describe the inverse of $I$. 
Let $\Phi: S^p_0\To S^q$ be a right-centralizer for which we may assume (and do) that $\Phi(f)=\Phi(f)e$ for every $f\in S^p_0$ when $e\in B$ is the initial projection of $f$.

% Here, $p,q\in(0,\infty]$ can be arbitrary; in particular we are not assuming $p<q$.
Fixing a norm one $\eta\in\mathscr H$, we see that $\Phi(\eta\otimes y)=\eta\otimes \phi$ for some $\phi\in H$ depending on $y$ (and $\eta$). Taking $\phi=\phi_\eta(y)$ we obtain a self-map on $\mathscr H$ which is easily seen to be quasilinear. Let $\zeta$ be another normalized vector in $\mathscr H$ and define $\phi_\zeta$  by the identity $\Phi(\zeta\otimes y)=\zeta\otimes\phi_\zeta(y)$. Let $u\in B$ be an isometry of $\mathscr H$ sending $\zeta$ to $\eta$, so that $(\eta\otimes y)u= u^*(\eta)\otimes y= \zeta\otimes y$. One has
$$%\begin{align*}
\|\phi_\zeta(y)- \phi_\eta(y)\|=
\|\eta\otimes (\phi_\zeta(y) - \phi_\eta(y))\|_q=
\|\Phi((\eta\otimes y)u)-(\Phi(\eta\otimes y)u)\|_q\leq C[\Phi]\|\eta\|\|y\|\|u\|_B.
%\leq M\|y\|.
$$%\end{align*}
Therefore, $\phi_\eta\approx\phi_\zeta$, with $\dist(\phi_\eta,\phi_\zeta)\leq C[\Phi]$ and so $\phi=\phi_\eta$ works in (\ref{times}).

The statement (a) is obvious. 
To prove (b) let us first check that $J$ is correctly defined. 
%Let us check that this correspondence works fine with classes.
Suppose $\Phi_1$ and $\phi_1$ satisfy an estimate 
$$%\begin{equation}\label{times}
\|\Phi_1(x\otimes y)-  x\otimes \phi_1(y)\|_q\leq M_1\|x\|\|y\|\quad\quad(x,y\in\mathscr H).
$$%\end{equation}
If $\Phi_1$ is equivalent to $\Phi$, then $\Phi_1=\Phi+\alpha+\beta$, where $\alpha$ is a morphism of right modules and $\beta$ is bounded. By Lemma~\ref{mor}(b), $\alpha$ is implemented by linear endomorphism of $\Hi$ and so $\alpha(f)=\ell\circ f$ for some fixed $\ell$ and all $f$. In particular $\alpha(x\otimes y)= x\otimes \ell(y)$, so
$$%\begin{equation}\label{times}
\|\Phi(x\otimes y)+  x\otimes \ell(y)-  x\otimes \phi_1(y)\|_q\leq M_2\|x\|\|y\|\quad\quad(x,y\in\mathscr H),
$$
where $M_2=M_1+\|\beta\|$. It follows that $\phi\approx\phi_1-\ell$ and so $\phi$ and $\phi_1$ share class in $\mathscr Q(\Hi)^\sim$.

The linearity of $J$ is now clear: assume $\Phi_i$ and $\phi_i$ satisfy estimates 
$$%\begin{equation}\label{times}
\|\Phi_i(x\otimes y)-  x\otimes \phi_i(y)\|_q\leq M_i\|x\|\|y\|\quad\quad(x,y\in\mathscr H),
$$
for $i=1,2$,
so that $J[\Phi_i]=[\phi_i]$. Then, if $c_i$ are complex numbers, one has
$$%\begin{equation}\label{times}
\|(c_1\Phi_1+c_2\Phi_2)(x\otimes y)-  x\otimes(c_1\phi_1+c_2\phi_2)(y)\|_q\leq M \|x\|\|y\|,
$$
with $M$ independent on $x,y\in\Hi$. %It is pretty obvious that $J$ is linear and it is even a morphism of two-sided modules when XXX
\end{proof}

\subsection{A natural isomorphism}
We are now ready for the main result of the Section.

\begin{theorem}\label{the:spatial} Consider the ``spatial part'' map $J:\mathscr C(S^p_0,S^q)^\sim_B\To \mathscr Q(\Hi)^\sim$ defined in Lemma~\ref{lem:spatial}.
\begin{itemize}
\item[(a)] If\: $0<p<\infty$ and\:  $p\leq q\leq \infty$, then $J$ is surjective.
\item[(b)] If\: $0<p< q\leq \infty$, then $J$ is moreover an isomorphism.
\end{itemize}
\end{theorem}

% Before going into the details let us take a look at the centralizers introduced in Sections \ref{sec:SpK}--\ref{sec:S2Pisier}.

\begin{proof}
The proof is mostly an assembly of previous results.
\medskip

(a) Since the inclusion of $S^p$ into $S^q$ is a (contractive) homomorphism it suffices to stablish the result for $q=p$.

Let us begin with the case $p=2$. Let $\phi$ be quasilinear on $\Hi$, set $T=\Hi\oplus_\phi\Hi$ and let
$$
\begin{CD}
0@>>> S^2 @>\jmath_\circ >>\Pi_\gamma(\Hi, T) @>\varpi_\circ >> S^2 @>>> 0
\end{CD}
$$
be extension provided by Proposition~\ref{prop:Pis}.
We want to see that $\phi$ is the spatial part of any centralizer $\Phi: S^2_0\To S^2$ representing that extension. Recall that we can construct such a $\Phi$ as $B-\Lambda$, where $B$ is a bounded section of $\varpi_\circ$ and $\Lambda: S^2_0\To \Pi_\gamma(\Hi, T)$ is a  morphism or right modules such that $(\varpi_\circ)\Lambda$ is the identity on $S^2_0$. Clearly, we may take $\Lambda(f)=L\circ f$, where $L:\Hi\To T$ is given by $L(y)=(0,y)$. As for $B$, we have no explicit description of $B(f)$ in general, which would require to know an explicit Pisier's lifting. However, if $f=\xoy$ has rank-one, one can always take $B(f)=x\otimes(\phi(y),y)$ since $\|(\phi(y),y)\|_\phi=\|y\|_\Hi$ so for  
$\Phi=B-\Lambda$ we do have $\Phi(\xoy)=x\otimes\phi(y)$ and the spatial part of $\Phi$ is $\phi$.

\medskip
Case $0<p<2$. Take $s$ so that $p^{-1}=2^{-1}+s^{-1}$ and apply Proposition~\ref{version} to the centralizer $\Phi:S^2_0\To S^2$ just obtained to conclude that the map $\Phi_{(s)}:S^p_0\To S^p$ defined by% such that
$$
\Phi_{(s)}(f)= \big{(}\Phi(v|f|^{p/2}) \big{)} |f|^{p/s} \quad\quad(v\text{ is the phase of } f),
$$
is a centralizer. We claim that $
\Phi_{(s)}$ has the same spatial part as $\Phi$. Indeed, if $x,y\in\Hi$ are normalized, then the phase of $\xoy$ is $\xoy$ itself and $|\xoy|^\alpha=x\otimes x$ for every $\alpha>0$, so
$$
\Phi_{(s)}(\xoy)=\Phi(\xoy)(x\otimes x)= (x\otimes \phi(y))(x\otimes x)= (x\otimes \phi(y)).
$$

Case $2<p<\infty$. Take $s$ so that $2^{-1}=p^{-1}+s^{-1}$ and apply Proposition~\ref{prop:cova1} to $\Phi$ get a centralizer $\Phi^{(s)}:S^{p}_0\To S^p$ satisfying
$$
\|\Phi(gf)-(\Phi^{(s)} g)f\|_2\leq M\|g\|_p\|f\|_s\quad\quad(g,f\in\frak F).
$$ 
If $\varphi$ is the spatial part of $\Phi^{(s)}$ we can clearly assume that $\Phi^{(s)}(\xoy)=x\otimes\varphi(y)$ for all $x,y\in\Hi$. If $x$ and $y$ are normalized, applying the preceding estimate with $g=\xoy, f=x\otimes x$ we have $gf=\xoy$ and
$$
\|\Phi(\xoy)-(x\otimes\varphi(y))(x\otimes x)\|_2\leq M,%\quad\quad(g,f\in\mathfrak F).
$$
 hence $\|\phi(y)-\varphi(y)\|_\Hi\leq M$ for every norm one $y\in\Hi$. This shows that $\varphi$ is strongly equivalent to $\phi$ and so $J[\Psi]=[\phi]$, which completes the proof of (a).

\medskip

(b) It only remains to see that $J$ is injective for $0<p<q\leq\infty$.
 Since $J$ is linear it suffices to check that if $\Phi: S^p_0\To S^q$ is a centralizer whose spatial part is trivial, then $\Phi$ is itself trivial -- as a centralizer. Assume then (\ref{times}) and that $\phi=\ell+\beta$, with $\ell:\Hi\To\Hi$ linear and $\beta$ bounded. Replacing $\Phi$ by $\Phi-\ell_\circ$ we obtain an equivalent centralizer whose spatial part is bounded (it is $\beta$, in fact). We will show that a centralizer with bounded spatial part has to be bounded. Note that $\Phi: S^p_0\To S^q$ has bounded spatial part if and only if one has $\|\Phi(x\otimes y)\|_q\leq M\|x\otimes y\|_p$ for some constant $M$ and every $x,y\in\Hi$.

First consider the case where $p<1$. Pick $f\in S^p_0$ and choose a Schmidt expansion, say $f=\sum_{n}s_nx_n\otimes y_n$. Then since the sequence $(x_n\otimes y_n)$ is isometrically equivalent to the unit basis of $\ell^p$ and $S^q$ is a $r$-Banach space for $r=\min(1,q)$, with $r>p$, we have
\begin{equation}\label{check}
\left\|\Phi(f)-\sum_n s_n\Phi(x_n\otimes y_n)\right\|_q\leq M_1\|f\|_p
\end{equation}
for some constant $M_1$ depending only on $\Phi, p$ and $q$: indeed it follows from the inequality in \cite[Lemma~3.4]{k} that one may take
$$M_1=\left( \sum_{k=1}^\infty \left(\frac{2}{k}\right)^{r/p} \right)^{1/p}Q(\Phi).$$

Now, if $p\geq 1$ we can use Proposition~\ref{version} again to lower $\Phi$ to a centralizer defined on $S^{1/2}$, say. So, take $s$ such that $p^{-1}+s^{-1}=2$ and let  $q_1$ be given by $q_1^{-1}=p^{-1}+s^{-1}$. We know from Proposition~\ref{version} and the Proof of Part (a) that the map $\Phi_{(s)}:S^{1/2}_0\To S^{q_1}$ defined by
$$
\Phi_{(s)}(h)=\Phi(v|h|^\frac{1}{2p})|h|^\frac{1}{2s}\quad\quad(h=v|h| \text{ is the polar decomposition})
$$
is a right-centralizer with the same spatial part as $\Phi$. But $
\Phi_{(s)}$ is bounded and so is $\Phi$.
\end{proof}

Let us take a look at the constructions of Sections \ref{sec:SpK}--\ref{sec:S2Pisier} in the light of the preceding Theorem. First of all, Corollary~\ref{cor:tildephi} describes, up to strong equivalence,  all centralizers in $\mathscr C(S^p_0, S^q)_B$ when $0<p<2$ and $q>p$. We suspect that $\tilde{\phi}$ is a centralizer as long as $0<p<q$, but we have been unable to prove it. Note that this is indeed de case for $p=2, q=\infty$ in view of  Section~\ref{sec:SpK}.

The condition $p<q$ cannot be removed from Part (b). Actually all bicentralizers $S^p_0\To S^p$ have bounded spatial part; see Section~\ref{sec:bicentralizers}. Incidentally, this implies that the self-extensions of $S^p$ occurring in the Proof of Theorem~\ref{the:spatial} (a) are quite different from those previously known.%, based on bicentralizers.

\section{Minimal extensions and $\mathcal K$-spaces}\label{sec:Kspaces}
Recall that a (complex) quasi Banach space $X$ is said to be a $\mathcal K$-space if every minimal extension (of quasi Banach spaces)
$
0\To \C\To Z\To X\To 0
$
splits. Equivalently, if for every dense subspace $X_0$ of $X$ and every quasilinear map $\varphi:X_0\To \mathbb C$ there is a linear map $\ell:X_0\To\mathbb C$ such that $\dist(\varphi,\ell)<\infty$.  The main examples of $\mathcal K$-spaces were discovered by Kalton and coworkers: it turns out that $\ell^p$ (or $L^p$) is a $\mathcal K$-space if and only if $p\in (0,\infty]$ is different from 1. See \cite{rib, k, rob, kalrob}. In contrast to the commutative situation, one has:

\begin{theorem}\label{kspace}
If\: $0<p<1$, then  $S^p$ is not a $\mathcal K$-space.% for no $p\in(0,1)$.
\end{theorem}

\begin{proof}
Let $\phi$ be quasilinear on $\mathscr H$ and let $\tilde\phi:S^p_0\To S^1$ be the right centralizer given by Corollary~\ref{cor:tildephi}. Composing with $\tr:S^1\To\mathbb C$ we get a quasilinear function $\varphi: S_0^p\To\C$ such that
$$\varphi(x\otimes y)=\tr(\tilde\phi(x\otimes y))=\tr(x\otimes \phi(y))=\langle \phi(y)|x\rangle.$$
Suppose there is a linear $\ell: S_0^p\To\C$ at finite distance from $\varphi$. As $\varphi(x\otimes y)\To 0$ for fixed $y$ when $x\To0$ in $\mathscr H$ the same occurs to  $\ell(x\otimes y)\To 0$  and, by Lemma~\ref{mor}(d), there is a linear map $L$ on $\mathscr H$ such that
$\ell(x\otimes y)=\langle L(y)|x\rangle$. This obviously implies $\dist(\phi,L)<\infty$. Starting with a non-trivial $\phi$ we get a non-trivial, minimal extension of $S^p$.
\end{proof}

Of course $S^1$ is not a $\mathcal K$-space as it contains a complemented subspace isomorphic to $\ell^1$, while $S^p$ is a $\mathcal K$-space for $p\in(1,\infty)$, as  all $B$-convex spaces are. Whether or not the spaces $K$ and $B$ are ``themselves'' $\mathcal K$-spaces is a fascinating mistery. % Kalton repeteadly conjectured that a Banach space is a $\mathcal{K}$-space if and only if it does not contain $\ell^1_n$ uniformly complemented; \cite[p. 815]{kalrob}, \cite[p. 11]{k-mah}, \cite[Problem 4.2]{k-handbook}.

We finally add a result which partially answers a question raised by Kalton and Montgomery-Smith at the end of their survey \cite[p. 1172]{kal-mon}.

\begin{proposition}\label{answer}
Let $\Phi:S^2_0\To L(\mathscr H)$ be a left centralizer from $S^2_0$ to $S^2$. Then the function $\varphi:S^1_0\To\C$ given by
\begin{equation}\label{vphi}
\varphi(f)=\tr\left( u|f|^{1/2} \Phi( |f|^{1/2})  \right),%$u$ is the phase of $f$
\end{equation}
 where $u$ is the phase of $f$, is quasilinear. Every quasilinear (complex) function on $S^1_0$ is at finite distance from one arising in this way.
\end{proposition}

\begin{proof}%[Sketch of the Proof]
Let us see the first part assuming that $\Phi$ takes values in $S^2$. A specialization
($p_2=q_2=s=2; p_1=q_1=1$) of the obvious left version of Proposition~\ref{version} shows that the map $\Phi_{(2)}:S^1_0\To S^1$ defined by $\Phi_{(2)}(f)= u|f|^{1/2} \Phi( |f|^{1/2}) $ is a
centralizer, hence a quasilinear map. Since the trace is bounded and linear on $S^1$, the composition $\varphi(f)
=\tr(\Phi_{(2)}(f))$ is quasilinear, too.

In any case, we know from Corollary~\ref{all} that there is a centralizer $\Psi:S^2_0\To S^2$ that induces an extension equivalent to that induced by $\Phi$. Hence (see Section~\ref{they}) there exist a morphism of left modules $\alpha:S^2_0\To L(\mathscr H)$ and a bounded homogeneous map $b:S^2_0\To S^2$ such that $\Phi=\Psi+\alpha+b$. We have
$$
\varphi(f)=\tr\left( u|f|^{1/2} \Psi( |f|^{1/2}) \right) + \tr\left( u|f|^{1/2} \alpha( |f|^{1/2}) \right)+\tr\left( u|f|^{1/2} b( |f|^{1/2}) \right).
$$
We have just proved that the first summand in the right-hand side of the preceding equality is a quasilinear function of $f$. The second one is linear since
$u|f|^{1/2} \alpha( |f|^{1/2})= \alpha( u|f|^{1/2} |f|^{1/2})=\alpha(f)$. The  third one is clearly bounded. Thus $\varphi$ is itself quasilinear.

As for the second part, let $\phi: S^1_0\To\C$ be a quasilinear function. Consider the map $S_0^2\times S^2_0\To\C$ sending $(f,g)$ to $\phi(fg)$. For fixed $g\in S^2_0$, the function $f\longmapsto \phi(fg)$ is quasilinear on  $S^2_0$, with constant at most $\|g\|_2Q(\phi)$. But, being a Hilbert space, $S^2$ is a $\mathcal K$-space and so there is a linear map $\ell_g:S^2_0\To \C$ (depending on $g$) such that
\begin{equation}\label{37}
|\phi(fg)-\ell_g(f)|\leq k\|g\|_2Q[\phi]\|f\|_2
\end{equation}
where $k\leq 37$ is the ``$\mathcal K$-space constant'' of $S^2$.

Next we want to see that $\ell_g(f)=\tr(L\circ f)=\tr(f\circ L)$ for some $L\in L(\mathscr H)$ depending on $g$. According to Lemma~\ref{mor}(d) it suffices to check that for each fixed $y\in \mathscr H$ one has $\ell_g(x\otimes y)\To 0$ as $x\To 0$ in $\mathscr H$. In view of (\ref{37}), it suffices to verify that for fixed $g\in S^2_0$ and $y\in \mathscr H$ one has
\begin{equation}\label{38}
\phi((x\otimes y)g)\To 0\quad\quad(\text{as }x\To 0).
\end{equation}
Write $g=\sum_{n=1}^mt_nx_n\otimes y_n$. Then
$$
(x\otimes y)g=g^*(x)\otimes y= \sum_{n=1}^mt_n\langle x|y_n\rangle x_n\otimes y.
$$
As $\phi$ is quasilinear we have the estimate (see the part of the argument marked with (*)  in   \cite[Proof of Lemma 3.2]{k})
$$
\left|\phi((x\otimes y)g)- \sum_{n=1}^mt_n\langle x|y_n\rangle \phi(x_n\otimes y) \right|
\leq Q(\phi) \sum_{n=1}^m\big{|}nt_n\langle x|y_n\rangle\big{|} \|x_n\|\|y\|
$$
and (\ref{38}) follows.

To sum up, there is homogeneous map $\Phi:S^2_0\To L(\mathscr H)$ such that
$$%\begin{equation}\label{necessarily}
|\phi(fg)-\tr(f\Phi(g))|\leq M\|f\|_2\|g\|_2\quad\quad(f,g\in S^2_0).
$$%\end{equation}
Clearly, $\phi\approx\varphi$, where $\varphi$ is given by (\ref{vphi}). It only remains to check that $\Phi$ is a centralizer. Take $g,f\in S^2_0, a\in B$. We have:
\begin{align*}
|\phi(f(ag))-\tr(f\Phi(ag))|&\leq M\|f\|_2\|ag\|_2,\\
|\phi((fa)g)-\tr(fa\Phi(g))|&\leq M\|fa\|_2\|g\|_2,
\end{align*}
so
$$
\|\Phi(ag)-a\Phi(g)\|_2=\sup_{\|f\|_2\leq 1}|\tr(f(\Phi(ag)-a\Phi(g))|\leq M\|a\|_B\|g\|_2
$$
and we are done.
\end{proof}

\section{Bicentralizers}\label{sec:bicentralizers}
 A bicentralizer is just a left centralizer which is also a right centralizer. 
 %Thus, a bicentralizer obeys an estimate
%$$
%\|\Omega(afb)-a\Omega(f)b\|_Y\leq C \|a\|_A\|f\|_X\|b\|_A\quad\quad(a,b\in A, f\in X).
%$$
Bicentralizers on the Schatten classes are the subject of \cite{k-trace} and \cite{kal-diff}.
It can be proved that every extension of quasi Banach $B$-bimodules $0\To S^q\To Z\To S^p\To 0$ arises from a bicentralizer $\Omega: S^p_0\To S^q$ although we will refrain from entering into the details here.
 Let us draw some consequences of the results proved so far.

\begin{theorem}
Let $\Omega:S^p_0\To S^q$ be a bicentralizer, with $p\neq q$. Then there exist $c\in\C$ and $M\geq 0$ such that $\|\Omega(f)-cf\|_q\leq M\|f\|_p$ for all $f\in S_0^p$. In particular, if $q>p$, then $\Omega$ is bounded.
\end{theorem}

\begin{proof}
Let us first observe that
\begin{equation}\label{eq:Obounded}
\|\Omega(\xoy)\|_q\leq M\|x\|\|y\|\quad\quad(x,y\in\Hi).
\end{equation}
To see this, we fix normalized $x_0,y_0\in\Hi$ and we set $\xi=\Omega(x_0\otimes y_0).$
But, if $x, y$ are normalized in $\Hi$, then there exist isometries $u,v\in B$ such that $y=v(y_0)$ and $x_0=u(x)$, hence $\xoy=v(x_0\otimes y_0)u$,
$$
\|v\xi u-\Omega(\xoy)\|_q\leq C(\Omega),
$$
and (\ref{eq:Obounded}) follows for some $M$ depending only on the modulus of concavity of $S^q$ and the numbers $\|\xi\|_q$ and $C(\Omega)$.
\medskip

Let us dispose of the case where $q>p$. As $\Omega$ is a right centralizer, we know from Lemma~\ref{lem:spatial} that there is a quasilinear map $\phi$ on $\mathscr H$ such that $\|\Omega(x\otimes y)-x\otimes\phi(y)\|_q\leq M\|x\|\|y\|$ for some $M$ independent on $x,y\in \mathscr H$. But $\Omega$ is also a left centralizer and so $\|\Omega(a(x\otimes y))-a\Omega(x\otimes y)\|_q\leq M\|x\|\|y\|$, which yields
$$
\|x\otimes\phi(ay)-x\otimes a\phi(y)\|_q=\|x\|\|\phi(ay)-a\phi(y)\|\leq M\|a\|_B\|x\|\|y\|\quad\quad(a\in B, x,y\in\mathscr H).
$$
As $\{ay: \|a\|_B\leq 1\}$ is the ball of radius $\|y\|$ in $\mathscr H$ we see that $\phi$ is bounded and so is $\Omega$; this was stablished during the proof of Theorem~\ref{the:spatial}(b).
\medskip

Case $q<p$. We know from Theorem~\ref{qlessp} that $\Omega$ is trivial as a right centralizer, so there is $L\in L(\Hi)$ such that 
$$
\|\Omega f- Lf\|_q\leq M\|f\|_p\quad\quad(f\in\mathfrak{F}).
$$
The map $f\longmapsto Lf$, being strongly equivalent to $\Omega$, is also a bicentralizer, so 
\begin{equation}\label{eq:irritating}
\|Laf- aLf\|_q\leq M\|a\|_B\|f\|_p\quad\quad(a\in B, f\in\mathfrak{F}).
\end{equation}%$$
We want to see that there is $c\in\C$ such that $L$ is, as a map from $S^p_0$ to $S^q$, strongly equivalent to $f\longmapsto cf$. Note that $L(\xoy)=x\otimes L(y)$, so 
(\ref{eq:Obounded}) implies that $L$ is a bounded operator on $\Hi$.
 It is a bit irritating that we cannot handle the case $q<1$ directly. So, let us first assume $1\leq q<p\leq\infty$, so that $S^q$ is locally convex.
 
After dividing $L$ by $M$ we infer from (\ref{eq:irritating}) that if $u\in B$ is unitary, then
$$
\|u^*Luf- Lf\|_q\leq \|f\|_p\quad\quad(f\in \mathfrak F).
$$
Let $U$ be the group of unitaries in $B$ that are compact perturbations of the identity, that is, $u\in U$ if and only if there is a $c\in\C$, with $|c|=1$, and $v\in K$ such that $u=c{\bf I}_\Hi+v$. This group is known to be amenable, so let $du$ be an invariant mean on $\ell^\infty(U,\C)$. Now let us treat $S^q$ as the dual of $S^r$, where $q^{-1}+r^{-1}=1$ and we ``average'' $L$ over $U$ taking the ``weak*-integral''
$$
\Lambda(f)=\int_U u^*Lufdu,\quad\quad\text{that is,}\quad\quad
\langle\Lambda(f),g\rangle= \int_U \langle u^*Luf, g\rangle du
$$
where $g\in S^r$ and the duality is given by the trace. Then
\begin{itemize}
\item The map $f\longmapsto\Lambda(f)$ is a morphism of right modules from $\mathfrak F$ to $S^q$ and so $\Lambda(f)=Vf$, where $V\in L(\Hi)$.
\item One has $u^*\Lambda(uf)=\Lambda f$, hence $Vuf=uVf$ for every $u\in U, f\in\mathfrak{F}$ and so $V=c{\bf I}_\Hi$.
\item For every $f\in\mathfrak F$, one has $\|\Lambda(f)-Lf\|_q\leq   \|f\|_p$.%
%, hence $\|V-L\|_{t}\leq 1$, where $q^{-1}=t^{-1}+p^{-1}$.
\end{itemize}
This completes the proof when $q\geq 1$.
\medskip

We finally consider the case $0<q<p<\infty, q<1$. We have to show that if $L\in B$ satisfies (\ref{eq:irritating}), then $f\in S^p_0\longmapsto Lf\in S^q$ is strongly equivalent to a multiple of the ``identity''. 
Treating the composition $f\longmapsto Lf$ as a bicentralizer  
we want to use Proposition~\ref{prop:cova1} to ``lift'' it to the locally convex zone. Let us adjust $s$ so that  $\Hom(S^s,S^q)_B=S^1$, that is, $q^{-1}=1+s^{-1}$ and we define $p_2$ by letting $p^{-1}=p^{-1}_2+s^{-1}$. 

Now we apply Proposition~\ref{prop:cova1} to obtain a right centralizer $\Gamma: S^{p_2}_0\To S^1$ such that
$$
\|Lgf-(\Gamma g)f\|_q\leq M\|g\|_{p_2}\|f\|_s\quad\quad(f,g\in\mathfrak{F}).
$$
In view of (a) it is clear that we may take $\Gamma g=Lg$ and so (b) guarantees that $g\longmapsto Lg$ is a bicentralizer from $S^{p_2}$ to $S^1$. It follows that there is a constant $c\in\C$ such that $\|Lg-cg\|_1\leq M\|g\|_{p_2}$ and so 
$$
\|Lgf-cgf\|_q\leq M\|g\|_{p_2}\|f\|_s\quad\quad(f,g\in\mathfrak{F}),
$$
which is enough.
\end{proof}

As for ``self-bicentralizers'' on $S^p$, we have the following extension of a result by Kalton. Here, $\ell^p_0$ stands for the finitely supported sequences of $\ell^p$.

\begin{theorem}\label{th:bic2}
Let $\phi:\ell^p_0\To\ell^p$ be a symmetric centralizer over $\ell^\infty$, with $p\in(0,\infty)$. Define a self map on $S^p_0$ as follows. Given $f\in S^p_0$ choose a Schmidt expansion $f=\sum_ns_nx_n\otimes y_n$. Let $(t_n)=\phi((s_n))$ and put $\Phi f=\sum_nt_nx_n\otimes y_n$. Then $\Phi: S^p_0\To S^p$ is a bicentralizer. Moreover, every bicentralizer on $S^p$ is strongly equivalent to one obtained in this way.
\end{theorem}

\begin{proof}[Sketch of the proof]
Symmetric means that there is a constant $M$ such that $|\phi(f\circ\sigma)-\phi(f)\circ\sigma|_p\leq M|f|_p$ for every $f\in\ell^p_0$ whenever $\sigma$ is a bijection of $\N$.

The proof required the following three facts:
\begin{enumerate}
\item The statement holds for $p>1$ as proved by Kalton in \cite[Theorem 8.3]{kal-diff}.

%\item Corollary~\ref{future} is true replacing right centralizer by bicentralizer everywhere.

\item The commutative versions of Proposition~\ref{prop:cova1} and \ref{version} hold: let $p,q,s\in (0,\infty)$ satisfy $p^{-1}=  q^{-1}+ s^{-1}$ and let $\psi:\ell^q_0\To\ell^q$ be a centralizer over $\ell^\infty$. Define $\psi_{(s)}:\ell^p_0\To \ell^p$ taking $\psi_{(s)}(f)=\omega(u|f|^{p/q})|f|^{p/s}$, where $u$ is the signum of $f$. Then $\psi_{(s)}$ is a centralizer and every $\ell^\infty$-centralizer on $\ell^p_0$ is strongly equivalent to one obtained in this way.

\item Referring to the preceding statement, $\psi_{(s)}$ is symmetric if and only if $\psi$ is.
\end{enumerate}
Now, let $\phi: \ell^p_0\To\ell^p$ be a symmetric $\ell^\infty$-centralizer, where $p\leq 1$. By (2) and (3), there is a symmetric centralizer $\psi$ on $\ell^2_0$ such that $\phi\approx \psi_{(s)}$, where $p^{-1}=2^{-1}+s^{-1}$ and we may assume $\phi=\psi_{(s)}$. Applying (1) to this $\psi$ we can ``extend'' it to a bicentralizer $\Psi:S^2_0\To S^2$ just taking
$$
\Psi(f)=\sum_nt_n x_n\otimes y_n,
$$
where $\sum_ns_nx_n\otimes y_n$ is the prescribed Schmidt expansion of $f$ and $\psi((s_n))=(t_n)$.
Finally, applying Proposition~\ref{version} to $\Psi$ with the same $s$ as before one obtains a bicentralizer $\Psi_{(s)}:S^p_0\To S^p$. This map is strongly equivalent to $\Phi$, from where it follows that $\Phi$ is a bicentralizer.

The ``moreover'' part follows from the case $p=2$, using again Proposition~\ref{prop:cova1} and \ref{version}.
\end{proof}

\section{Concluding remarks}

$\bigstar$ Most results in Sections~3 and 4 would generalize to noncommutative $L^p$ spaces associated to arbitrary von Neumann algebras as long as one could find a good substitute for Lemma~\ref{local}. More precisely, we ask if for every $\mathcal M$-centralizer $\Omega: L^p_0\To L^q$ with $0<q<p<\infty$ there is a system of trivial centralizers $\Omega_i$ such that $\dist(\Omega, \mathscr M_\mathcal M(L^p_0,L^q))=\sup_i \dist(\Omega_i, \mathscr M_\mathcal M(L^p_0,L^q))$. Here, $L^p_0=\{af^{1/p}:a\in \mathcal M\}$, where $f$ is a normal, faithful state on $\mathcal M$.

$\bigstar$ Concerning Theorem~\ref{kspace}, nobody knows if $K$ and $B$ are $\mathcal K$-spaces or not.
Kalton repeatedly conjectured an affirmative answer \cite[Problem 4.2]{k-handbook}, \cite[p. 11]{k-mah}, \cite[p. 815]{kalrob}.

There is a rather curious connection with Theorem~\ref{kspace}: if $K$ (or $B$) is a $\mathcal K$-space, then every quasilinear function $\varphi:S^p_0\To\mathbb C$ arises, up to a bounded perturbation, as $\varphi(f)=\sum_n s_n\langle \phi(y_n)|x_n\rangle$, where $\phi$ is a quasilinear map on $\mathscr H$.

Also, it seems to be interesting to determine if $L^p(\mathcal M)$ is a $\mathcal K$-space for $0<p<1$ if $\mathcal M$ is a von Neumann algebra with no minimal projection.

$\bigstar$ Proposition~\ref{answer} and the results of Section~\ref{sec:isom} imply that if $\phi:\ell^p_0\To\ell^p$ is a (not necessarily symmetric) centralizer over $\ell^\infty$ and $(e_n)$ is a fixed orthonormal basis in $\mathscr H$, then there is a left (or right, but not two-sided) centralizer $\Phi$ on $S^p_0$ such that $\Phi(\sum_ns_ne_n\otimes e_n)= \sum_nt_ne_n\otimes e_n$, where $(t_n)=\phi((s_n))$.

\section*{Acknowledgements}
I thank the anonymous referee of a previous version of this paper, who pointed out a serious error in Section 5, and Jes\'us Su\'arez de la Fuente, who suggested how to fix it.

\end{document}